\documentclass[12pt]{article}

\usepackage{fullpage}
\usepackage{amsmath,amsthm,amssymb,MnSymbol}
\usepackage{mathrsfs}
\usepackage{hyperref}
\hypersetup{colorlinks=true, linkcolor=red,citecolor=blue,urlcolor=blue}

\theoremstyle{theorem}
\newtheorem{theorem}{Theorem}[section]
\newtheorem{lemma}[theorem]{Lemma}
\newtheorem{proposition}[theorem]{Proposition}
\newtheorem{corollary}[theorem]{Corollary}

\newtheorem*{theorem*}{Question}

\theoremstyle{definition}
\newtheorem{definition}[theorem]{Definition}

\newtheorem{remark}[theorem]{Remark}

\newlength{\probwidth}
\newcommand{\N}{\mathbb{N}}
\newcommand{\C}{\mathcal{C}}
\newcommand{\PTIME}{\mathsf{P}}
\newcommand{\eval}{\mathit{eval}}
\newcommand{\Log}{\mathit{Log}_{>1}}

\renewcommand{\le}{\leqslant}
\renewcommand{\leq}{\leqslant}
\renewcommand{\ge}{\geqslant}
\renewcommand{\geq}{\geqslant}

\newcommand{\NEXP}{\mathsf{NEXP}}
\newcommand{\EXP}{\mathsf{EXP}}
\renewcommand{\P}{\mathsf{P}}

\newcommand{\NP}{\mathsf{NP}}
\newcommand{\PSPACE}{\mathsf{PSPACE}}
\newcommand{\Ppoly}{\mathsf{P/poly}}
\newcommand{\ioPpoly}{\mathsf{io}\text{-}\mathsf{P/poly}}

\newcommand{\T}{\mathsf{T}}
\newcommand{\PV}{\mathsf{PV}}
\newcommand{\V}{\mathsf{V}}
\newcommand{\U}{\mathsf{U}}

\renewcommand{\S}{\mathsf{S}}

\newcommand{\Bd}{\textit{Bd}}
\renewcommand{\succ}{\textit{succ}}
\newcommand{\start}{\textit{start}}
\newcommand{\Start}{\textit{Start}}
\newcommand{\Fail}{\textit{Fail}}
\newcommand{\Next}{\textit{Next}}
\newcommand{\Unique}{\textit{Unique}}

\newcommand{\meyer}{\mu}

\newcommand{\q}[1]{\textup{``}\textit{#1}\textup{''}}

\newcommand{\Lg}{\textit{Log}}

\renewcommand{\ldots}{...}

\newcommand{\io}{\textup{io-}}

\begin{document}

\title{\bf From Gödel incompleteness to the consistency of circuit lower bounds
}

\author{Albert Atserias\\\small Universitat Polit\`{e}cnica de
  Catalunya\\[-1ex]\small and Centre de Recerca Matem\`atica \\[-1ex] \small
  Barcelona, Spain\\[-1ex]\small\url{atserias@cs.upc.edu} \and Moritz
  M\"uller\\\small Universit\"at Passau\\[-1ex] \small Passau,
  Germany\\[-1ex]\small\url{moritz.mueller@uni-passau.de} }

\maketitle

\begin{abstract} 
  We prove that the bounded arithmetic theory~$\S^1_2$ is consistent
  with~$\EXP\not\subseteq\Ppoly$. More generally, we show that certain
  separations of~$\V^1_2$ from a theory~$\T$ imply the consistency
  of~$\T$ with~$\EXP\not\subseteq\Ppoly$. For~$\T=\S^1_2$, Takeuti
  (1988) established such a separation using a variant of G\"odel's
  consistency statement.  Analogous results hold
  for~$\PSPACE\not\subseteq\Ppoly$ but the required separations of
  theories are yet unknown.  Finally, we give magnification results
  for the hardness of proving almost-everywhere versions of these
  lower bounds.
\end{abstract}

\section{Introduction}

\noindent\textbf{Bounded arithmetic and computational complexity.} 
Bounded arithmetics
$$\textstyle \S^1_2\subseteq \T^1_2\subseteq
\S^2_2\subseteq\T^2_2\subseteq\cdots\subseteq \T_2=\bigcup_i\T^i_2$$
go back to Buss \cite{bussthesis} and are given by induction
principles for bounded formulas corresponding to the levels of the
polynomial hierarchy. E.g.,~$\T^1_2$ has induction
for~$\NP$-properties or, more precisely, for~$\Sigma^b_1$-formulas, which
define precisely the sets in~$\NP$ (in the standard
model). Informally, the theories are meant to capture reasoning with
concepts and functions of a certain computational
complexity. E.g.,~$\S^1_2$ formalizes polynomial-time reasoning.

These theories are obviously not foundational for the whole of
mathematics but they are for computer science: already low levels
formalize a large part of contemporary complexity theory.  For example in
 circuit complexity, certain central upper \cite{at} and lower~\cite{mp} bounds
have been formalized in (a certain
theory below)~$\T^2_2$. 

The first and final words in Kraj\'i\v{c}ek's 1995 monograph \cite{krabuch} ask for consistency results for  notoriously open complexity theoretic conjectures. Already consistency with low levels of the hierarchy is seen as precise evidence for the truth of the conjecture. This is one of the main motivations for bounded arithmetic  from the  perspective of computational complexity.

Concerning circuit lower bounds, the central open question is
whether~$\S^1_2$ can be proved consistent with super-polynomial
circuit lower bounds for~$\NP$, i.e.,~$\NP\not\subseteq\Ppoly$.

\bigskip\noindent\textbf{Bounded arithmetic and mathematical logic.}
Such consistencies can be inferred from a better understanding of
independence. For example, the consistency of~$\T^i_2$ with~$\NP\not\subseteq\Ppoly$ is
implied by~$\T^i_2\neq \T_2$. In fact, the separation of
theories~$\T^i_2 \not= \T_2$ is \emph{equivalent} to the consistency
of~$\T^i_2$
with~$\Sigma^\P_{i+1}\not\subseteq\textsf{co-$\Sigma^\P_{i+1}$/poly}$;
see~\cite{kpt}, also \cite[Section 10.2]{krabuch}.
%

Using G\"odel-type consistency statements,
Takeuti~\cite[p.104]{takeuti} proved the following separation (see also
\cite[Corollary~2.7]{kraexp}):

\begin{theorem}[Takeuti] \label{thm:takeuti} There is a bounded formula which is provable in $\V^1_2$ but not in~$\S^1_2$.
\end{theorem}

The theory~$\V^1_2$ is strong: it extends~$\T_2$ with a second sort of
{\em set variables}, and has induction for~$\NEXP$-properties or, more
precisely, for~$\Sigma^{1,b}_1$-formulas, which define precisely the sets
in~$\NEXP$ (in the standard model).  Informally,~$\V^1_2$ formalizes
exponential-time reasoning.
 
Not much more has been achieved. E.g., it is still unkown whether the
subtheory~$\U^1_2\subseteq\V^1_2$ that formalizes polynomial-space
reasoning is conservative over~$\S^1_2$.  G\"odel-type consistency
statements do not seem to reach very far -- see
\cite[Section~10.5]{krabuch} and \cite{feasinc} for discussions.  New
methods seem to be required and this is the main motivation for
bounded arithmetic from the perspective of mathematical logic.  In
Pudl\'ak's words from 1996 \cite{bottom}: ``we are not satisfied with
the current methods of proving independence results. The main reason
is that, except for Gödel's theorem which gives only special formulas,
no general method is known for proving independence of~$\Pi_1$
sentences.''

\bigskip\noindent\textbf{Known consistency results.}
Maybe not surprisingly, not too much is known concerning the
consistency of circuit lower bounds.  Some conditional consistency
results have been achieved in \cite{cokra}.  Unconditional ones have
been achieved in a line of works
\cite{KrajicekOliveira:Unprovability,BKO:Consistency,BydzovskyMuller:Ultrapowers,CKKO:LEARN}
but only for fixed-polynomial lower bounds.  These works rely on
witnessing theorems and complexity-theoretic methods (concretely, the
methods in~\cite{SanthanamWilliams:Uniformity}).

Super-polynomial lower bounds seem to require different
techniques. Only one is known~\cite{abm} (with some strengthenings
announced in \cite{thapen}), namely, that the theory~$\V^0_2$ is
consistent with~$\NEXP\not\subseteq\Ppoly$. Here~$\V^0_2$ is a
conservative extension of~$\T_2$ with set variables -- needed to speak
about (nondeterministic) exponential-time computations.
The proof is based on Ajtai's theorem \cite{ajtai} and its
improvements \cite{kpw,pbi}, and thereby on methods from mathematical
logic, in particular forcing \cite{riis,am,partialstr}.

\bigskip\noindent\textbf{Main contribution.} We show that Takeuti's theorem
does imply a consistency result  for super-polynomial circuit lower bounds,
namely the consistency of~$\S^1_2$ with~$\EXP\not\subseteq\Ppoly$.  At
first sight this might appear weak when compared to the just mentioned
result from \cite{abm}. This is not the case: while the formalization
of the
collapse~$\NEXP\subseteq\Ppoly$ requires set-sort quantification, the
formalization of the weaker collapse~$\EXP\subseteq\Ppoly$ can be
stated without such quantification, and for the number sort the
theory~$\S^1_2$ has considerable strength.  This follows from a well-known argument from the
  so-called Karp-Lipton Theorem for $\EXP$, stated in \cite{kl} and attributed to
  Meyer. Details follow.
  
The collapse of~$\EXP$ to~$\Ppoly$ would mean that every
exponential-time machine~$M$ is simulated by circuits of size at
most~$n^c$ for some~$c\in\N$. The direct formalization reads:
$$ 
\alpha^c_M:=\begin{array}{l}
              \text{``\em for all lengths~$n>1$ there exists a circuit~$C$ of size at most~$n^c$}\\
              \text{ \em such that for
              all~$x\in\{0,1\}^n$:\;\;~$C(x)=1$ iff $M$
              accepts~$x$.\em ''}
\end{array}
$$
Writing the statement that~$M$ accepts~$x$ requires an existential
set-sort quantifier for a computation of~$M$.  This can be avoided as
follows.  Consider an exponential-time machine~$M^*$ that,
given~$x\in\{0,1\}^n$ and~$i\in\N$, decides whether the~$i$-th bit of
(the binary code) of the computation of~$M$ on~$x$ is 1. From small
circuits simulating~$M^*$ we get a~$c\in\N$ such that:
$$
\meyer^c_M:=\begin{array}{l}
              \text{``\em for all lengths~$n>1$ there exists a circuit~$C$ of size at most~$n^c$}\\
              \text{ \em such that for
              all~$x\in\{0,1\}^n$:\;\;~$\textit{tt}(C_x)$ is a
              halting computation of~$M$ on~$x$.\em ''}
\end{array}
$$
Here,~$C_x$ is the circuit obtained from~$C$ by fixing the first~$n$
inputs to the bits of~$x$, and~$\textit{tt}(C_x)$ is the truth table
of the Boolean function it computes. Note that~$\meyer^c_M$ is a
number-sort sentence, i.e., it has no set variables; in
fact,~$\meyer^c_M$ is a~$\forall\Sigma^b_2$-sentence.

Letting~$M_1$ be a suitable universal exponential-time machine, either
of the two theories~$\{\neg\alpha^c_{M_1}\mid c\in\N\}$
or~$\{\neg\meyer^c_{M_1}\mid c\in\N\}$ is true (in the standard
model) if and only if~$\EXP\not\subseteq\Ppoly$.  We define
$$
\q{$\EXP\not\subseteq\Ppoly$}:=\{\neg\meyer^c_{M_1}\mid
c\in\N\},
$$ 
pronounced as ``the $\meyer$-formalization of $\EXP \not\subseteq
\Ppoly$''.  The $\alpha$-formalization of $\EXP \not\subseteq \Ppoly$
is defined analogously, using the $\alpha$-formulas instead of the
$\meyer$-formulas.  

We arrive at the main theorem:

\begin{theorem}\label{thm:maincon}\
\begin{enumerate}\itemsep=0pt
\item[(a)] $\S^1_2$ is consistent with
  $\big\{ \neg\meyer^c_{M_1}\mid c\in\N \big\}$, the
  $\meyer$-formalization of $\EXP \not\subseteq \Ppoly$.
\item[(b)] $\S^1_2(\alpha)$ is consistent with $\big\{\neg\alpha^c_{M_1}\mid c\in\N\big\}$, the $\alpha$-formalization of $\EXP \not\subseteq \Ppoly$.
\end{enumerate}
\end{theorem}

\noindent In~(b), the theory~$\S^1_2(\alpha)$ is a standard
two-sorted variant of~$\S^1_2$.
The result is robust in that it does not depend on the parti\-cular
definition of~$M_1$. Indeed, the consistencies hold for \emph{any}
machine~$M_1$ that in a certain precise sense
is~$\S^1_2(\alpha)$-provably universal for exponential time -- what
this means is that the theorem holds for \emph{any} machine that
satisfies Lemma~\ref{lem:M1} below. Furthermore, both formalizations
are implied over~$\S^1_2(\alpha)$ by corresponding lower bound
statements for {\em any} exponential-time machine~$M$; see
Theorem~\ref{thm:equiv} below.

\bigskip\noindent\textbf{Methods.}
The proof of (a) is based on proof-theoretic arguments, namely
the ``new style'' witnessing theorem
from~\cite{bb}: from a~$\V^1_2$-proof of a bounded
formula~$\varphi(x)$ one constructs an exponential-time machine~$M$
such that even~$\S^1_2(\alpha)$ can infer~$\varphi(x)$
{\em from} a computation of~$M$
on~$x$. The theory~$\S^1_2(\alpha)$ is, however, agnostic concerning
the existence of such computations. On the other
hand,~$\meyer^c_{M_1}$ implies that such computations exist by
providing small circuits describing them.  This way, one infers (a)
from Takeuti's Theorem \ref{thm:takeuti}. In fact, the sketched
argument does not rely on Takeuti's specific formula: {\em every}
model of~$\S^1_2$ that violates {\em any} bounded consequence
of~$\V^1_2$ satisfies \q{$\EXP\not\subseteq\Ppoly$}; see
Theorem~\ref{thm:congammaM} below.

The proof of~(b) is not a straightforward consequence of~(a). 
 One difficulity is that, while the equivalence
of the~$\alpha$- and~$\meyer$-formalizations is, as we show, provable
in~$\T^1_2(\alpha)$, for all we know it is not provable
in~$\S^1_2(\alpha)$. We bypass this by proving~(b) anyway with a kind
of model-theoretic argument dubbed ``simulating comprehension''
in~\cite{abm}. Intuitively,~$\alpha_{M_1}^c$ implies that the family
of sets given by circuits is rich -- rich enough to
ensure~$\T^1_2(\alpha)$; see Lemma~\ref{lem:alphagammaproof}
below.  Once~$\T^1_2(\alpha)$ is available,
the~$\T^1_2(\alpha)$-provable equivalence of the~$\alpha$-
and~$\meyer$-formalizations kicks in, and (b) follows again from
Takeuti's Theorem. This argument, which uses the assumption
that~$\alpha^c_{M_1}$ is provable to
simulate~$\Pi^b_1(\alpha)$-induction, relies on the formally verified
model-checkers constructed in \cite{abm}.

\bigskip\noindent\textbf{Additional contributions.}
The only obstacle to strengthen Theorem~\ref{thm:maincon} to theories
above~$\S^1_2$, or to the lower bound theory
\q{$\PSPACE\not\subseteq\Ppoly$} analogously defined, is our lack of
understanding of logical independence:

\begin{theorem}\label{thm:nbsort}  Let $\T$ be a theory containing $\S^1_2(\alpha)$.
\begin{enumerate}\itemsep=0pt
\item[(a)] If~$\T$ does not prove all number-sort formulas that are
  provable in~$\forall\Sigma^{1,b}_1(\V^1_2)$,
  then~$\T\cup\q{$\EXP\not\subseteq\Ppoly$}$ is consistent.

\item[(b)] If~$\T$ does not prove all number-sort formulas that are
  provable in~$\forall\Sigma^{1,b}_1(\U^1_2)$,
  then~$\T\cup\q{$\PSPACE\not\subseteq\Ppoly$}$ is consistent.
\end{enumerate}
\end{theorem}

Note that the number-sort formulas in this theorem can have arbitrary
unbounded quantifiers. Here, $\forall\Sigma^{1,b}_1(\U^1_2)$ is the
set of universal closures of $\Sigma^{1,b}_1$-formulas proved by
$\U^1_2$.  This is still a strong theory -- it
properly contains $\V^0_2$ and it can
count~\cite[Lemma~5.5.14]{krabuch}.

Our consistency can be strengthened to the consistency of an almost
everywhere lower bound~$\EXP\not\subseteq\ioPpoly$; see
Corollary~\ref{cor:aecons} below.  Borrowing a term from circuit
complexity~\cite{beyond}, we prove {\em magnifications}: if certain
theories do not prove such lower bounds, neither do certain apparently
much stronger theories.  
We view this as a contribution to {\em
  Razborov's program}~\cite{razclb,razlower,razannals} to identify a
formally precise barrier in circuit complexity, namely a natural
theory that formalizes known circuit lower bounds but is incapable to
prove strong conjectured ones.

\begin{theorem}\label{thm:magn}\
\begin{enumerate}\itemsep=0pt
\item[(a)] If $\S^1_2$ does not prove \q{$\EXP\not\subseteq\ioPpoly$}, then neither does $\forall\Sigma^{1,b}_1(\V^1_2)$.
\item[(b)] If $\S^1_2$ does not prove \q{$\PSPACE\not\subseteq\ioPpoly$}, then neither does $\forall\Sigma^{1,b}_1(\U^1_2)$.
\end{enumerate}
\end{theorem}

The corresponding statements for \q{$\EXP\not\subseteq\Ppoly$} or \q{$\PSPACE\not\subseteq\Ppoly$} are void
(see Remark~\ref{rem:void}).
This has been overlooked in~\cite{abm}.

\bigskip\noindent\textbf{Outline.} Carrying out the arguments sketched in this
introduction requires a fair amount of technical
work. E.g.,~$\T^1_2(\alpha)$ should formalize
Meyer's argument, or~$\S^1_2(\alpha)$ should know in a certain sense
that~$M_1$ is universal.
We construct suitable machines in Section~\ref{sec:machines}.
Section~\ref{sec:formalizations} proves implications between
the~$\alpha$- and~$\meyer$-formalizations over~$\S^1_2(\alpha)$
or~$\T^1_2(\alpha)$. The sketched arguments are then carried out in
Section~\ref{sec:cons} where Theorem~\ref{thm:maincon} and
\ref{thm:nbsort} are proved.  Section~\ref{sec:magnification} infers
Theorem~\ref{thm:magn} as a corollary to the proofs. Finally,
Section~\ref{sec:concl} comments on the open question for the
consistency of~$\PSPACE\not\subseteq\Ppoly$.

\section{Preliminaries}

We assume the reader is familiar with the standard notation from
bounded arithmetic~\cite{bussthesis,hp,krabuch}.  Here, we recall some
facts and fix some less standard notation.

\subsection{Bounded arithmetic}

Bounded arithmetic theories~$\S^i_2$ and~$\T^i_2$ are written in Buss'
language containing the relation symbol~$x{<}y$ and the function
symbols~$0$,~$1$,~$x{+}y$,~$x{\cdot}y$,~$\lfloor
x/2\rfloor$,~$x{\#}y$,~$|x|$. The theories are given by a finite set
of universal axioms for the symbols in the language plus
length-induction for~$\Sigma^b_i$-formulas in~$\S^i_2$, and standard
induction for~$\Sigma^b_i$-formulas in~$\T^i_2$.  Full bounded
arithmetic is~$\T_2:=\bigcup_{i>0}\T^i_2$.
We do not distinguish these theories from their conservative extensions to Cook's language $\PV$ \cite{cook} which contains symbols for all polynomial-time computable functions and proves certain defining equations for them. 
By a {\em term} we mean a term in Buss' language. E.g., Cantor's pairing $\langle x,y\rangle$ is a term. 
We write $x\in\Log$ for $1{<}x \wedge \exists y\ x=|y|$. In such a context we write e.g.\ $2^x,2^{x^2}$ etc.~for $1\#y,y\#y$ etc.. 

\begin{remark} \label{rem:PV}  
A formula (in Buss' language) is~$\Delta^b_1$ in~$\S^1_2$
(i.e.,~$\S^1_2$-provably equivalent both to a~$\Sigma^b_1$- and
a~$\Pi^b_1$-formula) if and only if~$\S^1_2$ proves that it is
equivalent to a quantifier-free~$\PV$-formula. In particular~$\S^1_2$
proves length and standard induction for
quantifier-free~$\PV$-formulas \cite[Lemma 5.2.9]{krabuch}.
\end{remark}

\medskip

\noindent\textbf{Circuits.}
A circuit with~$s$ gates is coded by a number
below~$2^{10\cdot s\cdot|s|}$. There is a~$\PV$-function~$\eval(C,x)$
that (in the standard model) takes the encoding~$C$ of a circuit and
evaluates it on input~$x$.  This means that, if~$x<2^n$,
where~$n \leq |C|$ is the number of input gates of~$C$, then the input
gates are assigned the bits of the length-$n$ binary representation
of~$x$, and the internal gates of the circuit are evaluated bottom up
from inputs to outputs; we assume that~$\eval(C,x) = 0$ if~$x\ge 2^n$
or if~$C$ does not code a circuit.

It is convenient to have circuits that take finite tuples~$\bar
x=(x_1,\ldots, x_k)$ as inputs; formally, such a circuit has~$k$
sequences of input gates, the~$i$-th taking the bits
of~$x_i$. Again,~$\PV$ contains an evaluation function~$\eval(C,\bar
x)$; it outputs~$0$ if any~$x_i$ has length bigger than the length of
its allotted input sequence. 
We write~$C(\bar x)$ instead of~$\eval(C,\bar x)$. 

For a circuit~$C$ taking~$(\ell+k)$-tuples as inputs and
an~$\ell$-tuple~$\bar x$ we let~$C_{\bar x}$ be the circuit obtained
by fixing the first~$\ell$ inputs to~$\bar x$; it takes~$k$-tuples as
inputs. Formally,~$C_{\bar x}$ is 
~$\PV$-term 
with variables~$C$ and $\bar x$
and~$\S^1_2$ proves~$C_{\bar x}(\bar y) =C(\bar x,\bar
y)$ and~$C_{\bar x} \le C$.

\begin{lemma}\label{lem:circuit} For every quantifier-free $\PV$-formula $F(\bar x)$ there is $c\in\N$ such that
$$
\S^1_2\vdash\forall n{\in}\Log\ \exists C{<}2^{n^c}\ \forall \bar x{<}2^n\ (C(\bar x)=1 \leftrightarrow F(\bar x)).
$$
\end{lemma}

\noindent\textbf{Two-sorted theories} Besides {\em number variables} $x,y,\ldots$ as before, two-sorted theories have {\em set variables} $X,Y,\ldots$ and
a binary relation symbol~$x\in X$. Models have the form~$(N,\mathcal
Y)$ where~$N$ interprets the number sort symbols of Buss' language,
set variables range over~$\mathcal Y$, and~$\in$ is interpreted by a
subset of~$N\times\mathcal Y$. The standard model has~$N=\N$ and the
obvious interpretations,~$\mathcal Y$ as the finite subsets of~$\N$
and interprets~$\in$ by actual membership.  Since we shall use also
capital letters for number variables (e.g.,~when intended to denote a
circuit) we write quantifiers on set variables as~$\exists_2X$
and~$\forall_2X$, i.e., with an index 2.

Given a formula~$F(\bar X,\bar x,v)$ or a circuit~$C(v)$, we
use~$F(\bar X,\bar x,\cdot)$ and~$C(\cdot)$ to represent the sets
of~$v$'s satisfying~$F$ or~$C$. More precisely, for a
formula~$\varphi(X,\ldots)$ without set sort equalities~$X=Y$ we
define the formulas
$$
\varphi\big(F(\bar X,\bar x,\cdot),\ldots\big) \;\text{ and }\; \varphi\big(C(\cdot),\ldots\big)
$$ 
by replacing atomic subformulas $t\in X$ of $\varphi$ by $F(\bar
X,\bar x,t)$ and~$C(t)=1$, respectively.

The sets of formulas~$\Sigma^b_i(\alpha)$ and~$\Pi^b_i(\alpha)$ are
defined as before but now allowing atomic subformulas~$t\in X$ for
terms~$t$ and set variables~$X$. In particular, set-sort
equality~$X=Y$ and quantifiers on set variables are not allowed. 
Using the notation
$$
X<z \;:=\; \forall x \;(x\in X\to x<z),
$$ 
the two-sorted theories~$\S^i_2(\alpha)$ and~$\T^i_2(\alpha)$ are
given by length-induction and, respectively, induction
for~$\Sigma^b_i(\alpha)$-formulas plus
 the universal closures of
\begin{equation*}
\begin{array}{ll}
\textit{set-boundedness:}&\exists z\ X<z,\\
\textit{extensionality:} &  \forall z\ (z \in X\leftrightarrow z \in Y)\ \to \ X =Y,\\
\textit{$\Delta^b_1(\alpha)$-comprehension:}
&\exists_2 X\ \forall x\ (x \in X\leftrightarrow x<y\wedge\varphi(\bar X,\bar x,x))
\end{array}
\end{equation*}
for all~$\varphi \in \Delta^b_1(\alpha)$, where~$\Delta^b_1(\alpha)$
denotes the collection of all formulas that are equivalent to both
a~$\Sigma^b_1(\alpha)$-formula and a~$\Pi^b_1(\alpha)$-formula
provably in the subtheory of~$\S^1_2(\alpha)$ that does not
have~$\Delta^b_1(\alpha)$-comprehension.
Full two-sorted bounded arithmetic
is~$\T_2(\alpha):=\bigcup_{i>0}\T^i_2(\alpha)$.  
As before, we identify~$\S^1_2(\alpha)$ with a
conservative extension to the language~$\PV(\alpha)$ that contains
symbols~$f^{\bar X}(\bar x)$ for all polynomial-time algorithms with
oracles~$\bar X$ (represented by set variables) and proves defining
equations.

\begin{remark} \label{rem:PValpha}
A formula is~$\Delta^b_1(\alpha)$ if and only if~$\S^1_2(\alpha)$
proves that it is equivalent to a
quantifier-free~$\PV(\alpha)$-formula 
without set equalities. 
In particular, that~$\S^1_2(\alpha)$ proves $\Delta^b_1(\alpha)$-induction  means that it
proves  induction for such formulas.
\end{remark}
The~$\hat\Sigma^{1,b}_1$-formulas are those of the
form~$\exists_2 X\varphi$
for~$\varphi\in \Sigma^{1,b}_0:=\bigcup_i\Sigma^b_i(\alpha)$.
The~$\Sigma^{1,b}_1$-formulas are those in the closure
of~$\hat\Sigma^{1,b}_1$ under~$\vee$ and~$\wedge$ and bounded
number-sort quantification~$\exists x{<}t$ and~$\forall x{<}t$ for
terms~$t$ not involving~$x$.  The theories~$\V^0_2$ and~$\V^1_2$ are
given by~$\T_2(\alpha)$ plus~$\Sigma^{1,b}_0$-comprehension
and~$\Sigma^{1,b}_1$-comprehen\-sion, respectively. The
theory~$\U^1_2$ is given by~$\V^0_2$
plus~$\Sigma^{1,b}_1$-length-induction: for
every~$\varphi\in\Sigma^{1,b}_1$ the universal closure of
$$
\varphi(\bar X,\bar x,0)\wedge\forall y{<}|z|\ (\varphi(\bar X,\bar x,y)\to \varphi(\bar X,\bar x,y+1))\ \to\
\varphi(\bar X,\bar x,|z|).
$$

\begin{lemma}\label{lem:strict} 
For every $\Sigma^{1,b}_1$-formula $\varphi$ there is a 
$\hat\Sigma^{1,b}_1$-formula $\hat\varphi$ such that
$\U^1_2\vdash (\varphi \to \hat\varphi)$
 and $\S^1_2(\alpha)\vdash (\hat\varphi \to \varphi)$.
\end{lemma}

\begin{proof} The formula $\hat\varphi$ is built by induction on the syntactic structure of $\varphi$: it transforms subformulas of the form $\forall x{<}t \ \exists_2 X\ \psi(X,x,\ldots)$ to the form $\exists_2 X\ \forall x{<}t\ \psi(X_x,x,\ldots)$. Here, we write $X_x$ for $H(X,x,\cdot)$ where
$H(X,x,v):=\langle x,v\rangle\in X$. 
 Part (a) is an instance of $\Sigma^{1,b}_1$-collection provable in $\U^1_2$; see \cite[Lemma 5.5.10]{krabuch}. Part (b) 
 follows from $\Delta^{b}_1(\alpha)$-comprehension.
\end{proof}

The {\em number-sort} formulas are those without (free or bound) set
variables. It is not hard to see
that~$\S^i_2(\alpha)$,~$\T^i_2(\alpha)$, and~$\T_2(\alpha)$ prove the
same number-sort formulas as~$\S^i_2$,~$\T^i_2$, and~$\T_2$,
respectively. A way to see this is the following lemma (which holds
also for~$\T^i_2$).  For~$N\models\S^1_2$ let~$\C_N$ be the family of
all sets of the form~$\hat C:=\{a\in N\mid N\models C(a)=1\}$
where~$C$ ranges over~$N$ (equivalently, over the circuits in the
sense of~$N$). The model~$(N,\C_N)$ interprets~$\in$ by set
membership. Then:

\begin{lemma}\label{lem:CM} 
  For every~$i>0$ and~$N$, if~$N \models \S^i_2$,
  then~$(N,\C_N)\models \S^i_2(\alpha)$.
\end{lemma}

\begin{proof} That~$(N,\C_N)$ satisfies set-boundedness and
  extensionality is clear. We verify length-induction and
  comprehension for the appropriate collections of formulas. First we
  need to introduce some notation. Let~$\varphi(\bar x)$ be
  a~$\Sigma^{1,b}_0$-formula with parameters from~$(N,\C_N)$,
  i.e.,~$\varphi(\bar x)$ is obtained from a~$\Sigma^{1,b}_0$-formula
  by plugging set parameters from~$\C_N$ and number parameters
  from~$N$.  Define~$\varphi^*(\bar x)$ by replacing every
  subformula~$t\in\hat C$ by~$C(t)=1$ where~$t$ is a term (possibly
  with parameters from~$N$) and~$\hat C$ is a set parameter
  from~$\C_N$. Note that every set parameter in~$\varphi(\bar x)$
  becomes a number parameter in~$\varphi^*(\bar x)$ and
$$
(N,\C_N)\models\forall \bar x\ (\varphi(\bar x)\leftrightarrow\varphi^*(\bar x)).
$$
In particular, if~$\varphi(x)$ is a~$ \Sigma^b_i(\alpha)$-formula with
parameters from~$(N,\C_N)$, then~$\varphi^*(x)$ is
a~$\Sigma^b_i$-formula with parameters from~$N$. Thus,
length-induction for~$\varphi(x)$ follows from length-induction
for~$\varphi^*(x)$, which holds in~$N$
because~$N\models\S^i_2$. 

To verify comprehension, let~$a\in N$ and let~$\varphi(x)$
be~$\Delta^b_1(\alpha)$ with
parameters from~$(N,\C_N)$. Then~$\varphi^*(x)$ is~$\Delta^b_1$ with
parameters from~$N$, say~$\bar b$. Since~$N\models\S^1_2$, Buss' witnessing
theorem implies that~$\varphi^*(x)$ is equivalent in~$N$
to~$F(\bar b,x)$ for a quantifier-free~$\PV$-formula~$F(\bar
y,x)$. For~$n\in N$ bigger than~$1,|a|,|\bar b|$,
Lemma~\ref{lem:circuit} gives~$C\in N$ such that
$$
N\models\forall \bar y,y,x{<}2^n\ \big(C(\bar y,y,x) = 1\leftrightarrow x<y\wedge F(\bar y,x)\big).
$$
In $N$, choose $D:=C_{\bar b,a}$, so
$N\models\forall x\ D(x)=C(\bar b,a,x)$.
Then~$\hat D\in \C_N$
and~
\[
(N,\C_N)\models\forall x\ (x\in\hat D\leftrightarrow x<a\wedge\varphi(x)).\qedhere
\]
\end{proof}

\subsection{Explicit machines}\label{sec:premachine}

We use multitape (deterministic oracle Turing) machines as our model
of computation, say with~$k$ input tapes and~$\ell$ oracle tapes, and
some work tapes. The tapes have cells indexed~$0,1,\ldots$, and the
content of cell 0 is a fixed symbol marking the end of the tape -- it
is never overwritten by something else and heads scanning it do not
move left. It takes inputs~$\bar x = (x_1,\ldots,x_k)\in\N^k$. The
input length is~$|\bar x|:=|x_1|+\cdots+|x_k|$.  Input tapes are read-only, i.e.,
heads on input tapes write what they scan and do not move right when
scanning blank. A computation stops when reaching a halting state,
which is either accepting or rejecting.
For~$\bar X=(X_1,\ldots,X_{\ell})\in P(\N)^\ell$ we denote
by~$M^{\bar X}$ the machine with oracles~$\bar X$. When entering a
special query state the computation moves to one out of~$2^\ell$ many
states encoding whether~$q_i\in X_i$ for all~$i=1,\ldots,\ell$
where~$q_i$ is the query, i.e., the number written in binary on
the~$i$-th oracle tape.

A {\em partial time-$t$ space-$s$ query-$q$ computation
  of~$M^{\bar X}$ on~$\bar x$} is a sequence of~$t+1$ configurations
the first being the start configuration of~$M^{\bar X}$ on~$\bar x$,
every other being the successor of the previous one, repeating halting
configurations.  Every configuration has head positions less than~$s$
on work tapes, less than~$|\bar x|+2$ on input tapes, and less
than~$|q|$ on query tapes.  Using~$|q|$ (instead of~$q$) means, on the
formal level, that we allow only polynomially bounded queries.  Such a
configuration is determined by~$\max\{s, |\bar x|+2,|q|\}+1$ numbers,
one for the state and, for each~$j<\max\{s, |\bar x|+2,|q|\}$, a
number coding for each tape the content of cell~$j$ and a {\em
  position bit} indicating whether the current head position is to the
left of~$j$ or not.

We identify~$M$ with its code by a number and assume that the numbers
determining a configuration are less than~$M$. We code a configuration
as above by the set~$X$ of~$\langle i,j\rangle$
for~$i<\max\{s, |\bar x|+2,|q|\}+1$ and~$j<|M|$ such that the~$i$-th
number in the configuration has~$j$-th bit~$1$.  There
are~$\PV(\alpha)$-functions which, with such an~$X$ as oracle,
retrieve the state, the head positions, the symbols scanned, and the
queries in~$X$. We make no assumptions on what these functions return
when~$X$ is not coding a configuration as above.  The partial
computation is coded by a set~$Y$ such
that~$Y_i:=\{x\mid \langle i,x\rangle\in Y\}$ is the~$i$-th
configuration.

Note that, in the standard model, $X < \Bd^0(M,\bar x,s,q)$ and
$Y < \Bd(M,\bar x,t,s,q)$ where 
$$
\begin{array}{lclcl}
\Bd^0(M,\bar x,s,q)&:=&\langle \max\{s 
,  |\bar x|{+}2,|q|\}{+}1,|M|\rangle,\\
\Bd(M,\bar x,t,s,q)&:=&\langle  t,\Bd^0(M,\bar x,s,q)\rangle.
\end{array}
$$

We proceed to the formal level. Let~$k\in\N$ be the number of inputs,
let~$\ell\in\N$ be the number of oracles, let~$\bar x=(x_1,\ldots,
x_k)$ be a~$k$-tuple of number variables, and let~$\bar X=(X_1,\ldots,
X_\ell)$ be an~$\ell$-tuple of set variables.
The individual bits that encode successor configurations are
implicitly computed by the following oracle machine. Recall our
convention that the successor of a halting configuration is the
configuration itself. If any of the ``checks'' in Lines~1--3 in the
description of the machine fails, then the machine outputs ``fail'' in
Line~4; else it completes Lines~5--9 and outputs a bit.
 \begin{enumerate}\itemsep=0pt
\item[0.] inputs $(M,\bar x,s,q)$ and oracles $\bar X,X$.
\item check that $M$ is the code of an oracle machine with $k$ input tapes and $\ell$ oracle tapes.
\item check that the state, the head positions, the scanned symbols,
  and the queries encoded in~$X$ are not off; i.e., (a) the retrieved
  state is a state of~$M$,
  (b) the retrieved
  position of all oracle heads is less than~$|q|$, and (c) the
  retrieved position of all work heads is less than~$s$.
\item check that the new head positions as determined by the
  transition table of $M$ on the state, the scanned symbols, and the
  head positions encoded in $X$ are also not off.
\item if any of the above checks fails, output ``fail'' and halt; else:
\item query $X$ to retrieve the state, the head positions, the scanned
  symbols, and the queries, and then look up the transition table
  of~$M$ and, if required, query the oracles~$\bar X$ to compute the
  new state, the new cell contents, and the new head positions.
\item compute $i$ and $j$ such that $v=\langle i,j\rangle$.
\item if $i > \max\{s 
,  |\bar x|+2,|q|\}$, then output 0. 
\item if $i = \max\{s
,  |\bar x|+2,|q|\}$, then output the $j$-th bit of the new state (Line~5).
\item if~$i< \max\{ s
, |\bar x|+2,|q|\}$, then output the~$j$-th
  bit of the number~$a<M$ that for each tape codes the content of
  cell~$i$ (from oracle~$X$ or Line~5, as appropriate) and its
  position bit (determined by the new position of its head from Line~5).
\end{enumerate}
The machine runs in polynomial time and returns ``fail'' or a bit.
Note that it does not depend on~$v$ whether ``fail'' is returned or
not.  Let~$\succ^{\bar X,X}(M,\bar x,s,q,v)$ be
the~$\PV(\alpha)$-symbol for this machine.

Clearly, there is a
quantifier-free~$\PV(\alpha)$-formula~$\Start(M,\bar x,s,q,v)$ such that,
provably in~$\S^1_2$, the set defined by~$\Start(M,\bar x,s,q,\cdot)$ is
the start configuration of~$M$ on~$\bar x$ if~$M$ codes a suitable
machine. Similarly, there are
quantifier-free~$\PV(\alpha)$-formulas~$\Fail$ and~$\Next$ with
suitable free variables of both sorts which, provably
in~$\S^1_2(\alpha)$, determine if the algorithm
of~$\succ^{\bar X,X}(M,\bar x,s,q,v)$ outputs
``fails''~(formula~$\Fail$) and, otherwise determines the output
bit~(formula~$\Next$). We define
$$
\q{$Y$ is a partial time-$t$ space-$s$ query-$q$ computation of $M^{\bar X}$ on $\bar x$} 
$$
as the following 
$\Pi^b_1(\alpha)$ formula with free number variables $M,\bar x,t,s,q$ and free set variables~$Y,\bar X$:
$$
\begin{array}{l}
\forall i{\le}t\ \forall v{<}\Bd(M,\bar x,t,s,q)\ (v\in Y_i\to v<\Bd^0(M,\bar x,s,q))\\
\wedge\ \forall v{<}\Bd^0(M,\bar x,s,q)\ \big( v\in Y_0\leftrightarrow\Start(M,\bar x,s,q,v)\big)\\
\wedge\ \forall i{<}t\ \forall v{<}\Bd^0(M,\bar x,s,q)\ \big( v\in Y_{i+1}\leftrightarrow 
\Next(\bar X,Y_i,M,\bar x,s,q,v) \big)\\
\wedge\ \forall i{<}t \ \neg \Fail(\bar X,Y_i,M,\bar x,s,q).\\
\end{array}
$$
Here again, we write $Y_i$ for $H(Y,i,\cdot)$ where
$H(Y,i,v):=\langle i,v\rangle\in Y$. Note that the formula does not
state~$Y<\Bd(M,\bar x,t,s,q)$ 
because this would require unbounded universal quantifiers.

\begin{lemma}\label{lem:unique}
If $\Unique(\bar X,Y,Z,M,\bar x,t,s,q)$ denotes
the formula
\begin{equation*}
\begin{array}{l}
 \q{$Y$ is a partial time-$t$ space-$s$ query-$q$ computation of $M^{\bar X}$ on $\bar x$} \\ 
\ \wedge\ \q{$Z$ is a partial time-$t$ space-$s$ query-$q$ computation of $M^{\bar X}$ on $\bar x$} \\
\ \to\ \forall v{<}\Bd(M,\bar x,t,s,q)\ (v\in Y\leftrightarrow v\in Z),
\end{array}
\label{eqn:unique}
\end{equation*}
then $\T^1_2(\alpha) \vdash \Unique(\bar X,Y,Z,M,\bar x,t,s,q)$, and
$\S^1_2(\alpha) \vdash \Unique(\bar X,Y,Z,M,\bar x,t,|s|,q)$.
\end{lemma}


\begin{proof} In~$\T^1_2(\alpha)$ use~$\Pi^b_1(\alpha)$-induction to
  prove for all~$i\leq t$, that the~$i$-th configurations in~$Y$
  and~$Z$ are equal:
$$
\forall v{<}\Bd^0(M,\bar x,s,q) \ \big(\langle i,v\rangle\in Y\leftrightarrow  \langle i,v\rangle\in Z\big).
$$
Replacing~$s$ by~$|s|$, this is~$\Delta^b_1(\alpha)$-induction,
available in~$\S^1_2(\alpha)$; cf.~Remark~\ref{rem:PValpha}.
\end{proof}

There
are~$\PV(\alpha)$-functions~$\textit{time}^Y(t),\textit{space}^Y(t),\textit{query}^{Y}(t)$
which compute, respectively, the mini\-mal~$i\le t$ such that~$Y_i$ is
halting and~$i=t+1$ if all are non-halting, the maximal cell visited
on any work tape in any configuration, and the maximal cell visited on
any oracle tape in any configuration.  E.g., for~$\textit{space}^Y(t)$
we can assume that the machines mark visited cells; then binary search
can compute, for any work tape and time~$i$, the maximum cell visited
on it.

In~$\S^1_2$ we can prove, by~$\Delta^b_1(\alpha)$-induction, that
given a partial time-$t$ space-$s$ query-$q$ computation~$Y$
of~$M^{\bar X}$ on~$\bar{x}$, the head positions on the input tapes
are not off. If also time, space and query bounds are verifiable
in~$\S^1_2(\alpha)$ we talk of an explicit machine:

\begin{definition} Let~$M$ be an oracle machine with~$k \in \N$ input tapes and~$\ell \in \N$
  oracle tapes. The machine~$M$ is {\em explicit} if there are
  terms~$t_0(\bar x),s_0(\bar x),q_0(\bar x)$ such that
\begin{eqnarray*}
\S^1_2(\alpha)&\vdash& \q{$Y$ is a partial time-$t$ space-$s$ query-$q$ computation of $M^{\bar X}$ on $\bar{x}$}\\
&&\ \to \textit{time}^Y(t)<t_0(\bar x)\wedge \textit{space}^Y(t)<s_0(\bar x)\wedge \textit{query}^{Y}(t)<|q_0(\bar x)|.
\end{eqnarray*}
In this case we say that $M$ is {\em witnessed} by $t_0,s_0,q_0$ and write 
\begin{eqnarray*}
\Bd^0_M(\bar x)&:=&\Bd^0(M,\bar x,s_0(\bar x),q_0(\bar x))\\ 
\Bd_M(\bar x)&:=&\Bd(M,\bar x,t_0(\bar x),s_0(\bar x),q_0(\bar x)).
\end{eqnarray*}
The machine~$M$ is an {\em explicit~$\EXP$-machine} if there exists a
term~$t_M(x)$ such that the above holds for~$t_0(\bar x)=t_M(\bar x)$
and~$s_0(\bar x)=t_M(\bar x)$, and~$q_0(\bar x)=t_M(\bar x)$.  
The machine~$M$
is an {\em explicit~$\PSPACE$-machine} if there exists a term~$t_M(x)$
such that the above holds for~$t_0(\bar x)=t_M(\bar x)$ and~$s_0(\bar
x)=|t_M(\bar x)|$, and~$q_0(\bar x)=t_M(\bar x)$. 
The machine~$M$ is an
{\em explicit~$\P$-machine} if there exists a term~$t_M(x)$ such that
the above holds for~$t_0(\bar x)=|t_M(\bar x)|$ and~$s_0(\bar x)=|t_M(\bar
x)|$, and~$q_0(\bar x)=t_M(\bar x)$.  In all three cases we say that~$M$ is \emph{witnessed} by the term~$t_M(\bar x)$.
%
%
\end{definition}

\begin{remark} 
  For every term~$t(\bar x)$ there exist~$c \in \N$ such
  that~$\S^1_2$
  proves~$(|\bar x|>1 \to t(\bar x) \leq 2^{|\bar x|^c})$.  This means
  that the time and space bounds~$t_0(\bar x)$ and~$s_0(\bar x)$ of an
  explicit~$\EXP$-machine can be chosen to be of the
  form~$2^{|\bar x|^c}$ with~$c \in \N$. Similarly, the space bound of
  an explicit~$\PSPACE$-machine and the time bound of an
  explicit~$\P$-machine can be chosen of the form~$|\bar x|^c$.
\end{remark}

The theory~$\S^1_2(\alpha)$ proves that every function~$f^{\bar
  X}(\bar x)\in\PV(\alpha)$ is computed by an explicit
oracle~$\P$-machine. See \cite[Lemma~18]{abm} for a precise
statement. For example, there are explicit
oracle~$\P$-machines~$M_\succ$ and~$M_\start$, with
suitable numbers of inputs and oracles, which compute
the~$\PV(\alpha)$-function~$\succ$ and
the~$\PV(\alpha)$-function~$\start$ which indicates the truth
value of the quantifier-free formula~$\Start$.

For explicit machines we omit writing bounds $t,s,q$ in the formulas above. E.g., if $M$ be an explicit  $\PSPACE$-machine witnessed by $t_M(\bar x)$, then
the formulas
$$
\begin{array}{l}
\q{$Y$ is a partial time-$t$ computation of $M^{\bar X}$ on $\bar x$},\\
\q{$Y$ is a halting computation of $M^{\bar X}$ on $\bar x$}
\end{array}
$$ 
are obtained by substituting~$|t_M(\bar x)|$ for~$s$ and~$t_M(\bar
x)$ for~$q$.  The latter additionally substitutes~$t_M(\bar x)$
for~$t$.  Obvious modifications give formulas with ``accepting'' or
``rejecting'' instead of ``halting''.

\section{Formally verified universal machines}\label{sec:machines}

\subsection{A general universal machine} 
The 
universal machine takes input~$(M,\bar x,t,s,q)$ and simulates~$t$
steps of the machine~$M$ on~$\bar x$ within space~$s$ and query
bound~$q$.  We give a careful implementation so that~$\S^1_2(\alpha)$
can verify its correctness. Correctness must be formulated with care
because~$\S^1_2(\alpha)$ cannot prove that exponential-time
computations exist.


\begin{theorem}\label{thm:univ} 
  Let~$k,\ell\in\N$, let~$\bar x=(x_1,\ldots, x_k),M,t,s,q$ be number
  variables, and let~$\bar X=(X_1,\ldots, X_\ell),Y,Z$ be set
  variables.  There are~$c\in\N$, an explicit machine~$U$ with~$k+4$
  input tapes and~$\ell$ oracle tapes,
  and~$\Delta^b_1(\alpha)$-formulas~$F,G$ such that,
  writing~$\bar y := (M,\bar x,t,s,q)$ and~$n := |\bar y|$, the
  following hold:
 \begin{enumerate} \itemsep=0pt
 \item[(a)] $U$ is witnessed by $t_U,s_U,q_U$ where
   $t_U(\bar y) := t \cdot (s+n)^c$, $s_U(\bar y) := (s+n)^c$,
   $q_U(\bar y)=q$.
\item[(b)] $\S^1_2(\alpha)\vdash  \q{$Y$ is a partial time-$t$ space-$s$ query-$q$ computation of $M^{\bar X}$ on $\bar x$}$ \\
\hspace*{7ex} $\ \ \to \q{$F(Y,\bar y,\cdot)$ is an accepting computation  of $U^{\bar X}$ on $\bar y$}$\\
  \hspace*{10ex}
  $\ \ \wedge\ \forall v{<}\Bd(\bar y)\
  (G(F(Y,\bar y,\cdot),\bar y,v)\leftrightarrow v\in Y)$.
\item[(c)] $\S^1_2(\alpha)\vdash \q{$Z$ is an accepting computation of $U^{\bar X}$ on $\bar y$}$ \\
  \hspace*{7ex} $\ \ \to\q{$G(Z,\bar y,\cdot)$ is a partial time-$t$ space-$s$ query-$q$  computation  of $M^{\bar X}$ on $\bar x$}$\\
  \hspace*{10ex}
  $\ \ \wedge\ \forall v{<}\Bd_U(\bar x)\ (F(G(Z,\bar y,\cdot),\bar y,v)\leftrightarrow v\in Z)$.
\end{enumerate}
\end{theorem}

\begin{proof} 
  The machine~$U$ on input~$\bar y = (M,\bar{x},t,s,q)$ of length~$n =
  |\bar y|$ iterates the function~$\succ$ for~$t$ times using two {\em
    configuration tapes} holding configurations coded as binary
  strings.  This is done in~$t$ {\em rounds}, one for each~$u < t$.
  In round~$u>0$, the machine~$U$ overwrites one configuration tape,
  namely cell~$i+1$ with the bit~$\Next(\bar X,X,M,\bar x,s,q,i)$
  where~$X$ is the configuration on the other configuration tape,
  computed in the previous round~$u$; round~$0$ uses~$\Start(M,\bar
  x,s,q,i)$ instead. In each round,~$U$ also evaluates~$\Fail(\bar
  X,X,M,\bar x,s,q)$ and, in case it holds, sets a {\em flag bit}
  to~$1$; at the start the flag bit is set to~$0$. After round~$t$,
  the machine accepts if this flag bit is~$0$, and rejects if it
  is~$1$.  In the details that follow we define~$U$ in such a way that
  one can compute the exact times at which each round starts and ends,
  and how each new successor configuration is computed from the
  previous one, uniformly, and in exactly the same number of steps in
  each round.

\bigskip

\noindent\textbf{Description of~$U$.} The machine~$U$ starts with a {\em
    preparation phase} and ends with a {\em wrap-up phase}, both using
  explicit~$\P$-machines. In the preparation phase, it checks
  that~$M$ codes an oracle machine with~$k$ input tapes and~$\ell$
  oracle tapes, and halts rejecting if this is not the case. It
  computes~$s_0:=\Bd^0(M,\bar x,s,q)$ and writes it in binary on an
  own tape. Further, it computes the four
  strings~$0^{|s_0|},0^{|s_0|},0^{|t|},1^{|t_0|}$, each on an own
  tape. We refer to the first two as the {\em binary counters}, to the
  third as the {\em binary clock},  and to the fourth as the
  {\em unary clock}. Here,~$t_0=t_0(M,\bar x,t,s,q)$ is such that
  that~$\S^1_2 $ proves $ \forall v{<}s_0\; \max\{t_\start(M,\bar
  x,s,q,v),t_{\succ}(M,\bar x,s,q,v)\}\leq t_0$; the
  terms~$t_\start$ and~$t_\succ$ witness the explicit
  oracle~$\P$-machines~$M_\start$ and~$M_\succ$ which compute
  the predicate~$\Start$ and  the function~$\succ$. Thus,
  $M_\start$ and~$M_\succ$ run in time and space~$|t_0|$.

  We describe round~$u < t$, which starts with the binary expansion
  of~$u$ written on the binary counter~$0^{|t|}$. When~$u=t$ what we
  describe is the wrap-up phase. The machine~$U$ takes
  exactly~$2|t|+2$ steps to move forth on the counter until the first
  blank cell and then back to cell~$0$. On the pass forth it checks
  whether~$u$ equals the binary expansion of~$t$. It does this by
  moving simultaneously on the input tape which stores~$t$. On the
  pass back it updates the clock to~$0^{|t|}$ in case~$u=t$ and
  to~$u+1$ otherwise.  In the former case,~$U$ halts and accepts or
  rejects as described above. In the latter case,~$U$ computes the
  next configuration of~$M$ as follows.

  Case~$u>0$. The machine~$U$ skips cell~$0$ and fills cell~$i+1$
for~$i=0,1,\ldots,s_0-1$ of the configuration
  tape currently to be filled with the bits describing the successor
  configuration of~$X$, where~$X$ is the configuration stored on the
  other configuration tape.  This is done by first moving one cell
  right to skip cell~$0$, and then using the first binary
  counter~$0^{|s_0|}$ to fill cells~$1,2,\ldots,s_0$ as follows. The
  counter is updated in~$2|s_0|+2$ steps to~$0^{|s_0|}$ when its
  value~$i$ equals~$s_0$, and to $i+1$ otherwise. On the pass forth,
  the value~$i$ is copied  on a work tape. On the pass back, the
  copy tape is erased if~$i$ equals~$s_0$, and otherwise it is kept
  and its head is taken back to cell~$0$. In the former case,~$U$
  takes~$s_0+1$ additional steps to move the head of the configuration
  tape back to cell~$0$. Then it enters the next round~$u+1$ in one
  more step. In the latter case,~$U$ computes the
  predicates~$\Fail(\bar X,X,M,\bar x,s,q)$ and~$\Next(\bar X,X,M,\bar x,s,q,i)$ as described below. If the~$\Fail$ predicate does not hold,
  then~$U$ writes the bit indicating the truth value of the~$\Next$
  predicate on the configuration tape currently to be filled, and
  moves one cell right. If the~$\Fail$ predicate indicates failure,
  then~$U$ writes~$0$, sets the flag bit to~$1$ in the state, and also
  moves one cell right.

The computation of the predicates~$\Fail$ and~$\Next$ on~$i$ is done
by simulating the machine~$M_\succ$ on input~$(M,\bar x,s,q,i)$ 
and oracles~$\bar X,X$ using an extra array of work tapes. For the common
oracles~$\bar X$, it uses the oracle tapes of~$U$. For the oracle~$X$,
the simulation uses one of the work tapes as a virtual oracle tape,
and the relevant configuration tape as a virtual oracle as described
below. In order for subsequent simulations not to disturb one another
we assume that~$M_\succ$ cleans its tapes and moves all heads back to
cell~$0$ before halting. We explain now how one step of~$M_\succ$ is
simulated.

For each step of~$M_\succ$, the machine~$U$ moves one cell right on
the unary clock~$1^{|t_0|}$. When reading blank it takes~$|t_0|+1$
steps back to cell~$0$. When reading~$1$, the next step of~$M_{\succ}$
is carried out using the current contents of the work
tapes.
To simulate the next step,~$U$ needs to answer the potential query
to~$X$, say~$z\in X?$, where~$z$ is the query written on the virtual
oracle tape. This is done at every simulation step as follows: on the
configuration tape holding~$X$, the machine~$U$ moves one cell right
to skip cell~$0$, and then moves forth each step updating the second
binary counter~$0^{|s_0|}$ until it reaches position~$s_0+1$, to
finally move back to cell~$0$ resetting the counter
to~$0^{|s_0|}$. This takes exactly~$1+(2|s_0|+2)s_0+(s_0+1)$ many
steps. Each clock update additionally checks whether its current value
equals~$z$ -- then the scanned bit (i.e., cell~$z+1$ on the tape that
holds~$X$ because we skept cell~$0$) is the answer to the query~$z\in
X?$, which can be recorded in the state. Below we refer to this
exponential but ``simple'' computation that determines the answer to
the query~$z \in X?$ as the \emph{query-resolution computation}.

Case~$u=0$ is similar but instead of simulating~$M_\succ$ the machine
computes the bit indicating the truth value of~$\Start(M,\bar x,s,q,i)$ by
simulating~$M_\start$. Since~$M_\start$ has no oracles, we agree that
the~$z$ in the query~$z\in X?$ is always~$1$ and that the
query-resolution computations are executed nonetheless. This way the
execution of each round~$u<t$ takes exactly the same time,
whether~$u>0$ or~$u=0$. The wrap-up phase where~$u=t$ takes
time~$(2|t|+2)+1$ instead.

\bigskip

\noindent\textbf{Time.}  
Assume~$M$ codes a suitable machine - otherwise~$U$ halts in time
polynomial in~$|M|$.  One step of~$M_\start$ or~$M_\succ$ is simulated
by~$t_1:=1+(2|s_0|+2)s_0+(s_0+1) +1$ steps. Therefore, each simulation
of~$M_\start$ or~$M_\succ$ takes~$t_2:=1+|t_0|\cdot
t_1+(|t_0|+1)$ steps.  To fill one cell of the configuration tape
takes~$t_3:=(2|s_0|+2)+t_2+1$ steps.  Thus, round~$u<t$
takes~$t_4:=(2|t|+2)+1+s_0\cdot t_3+(s_0+1)+1$ steps. The wrap-up
phase where~$u=t$ takes~$q_0:=(2|t|+2)+1$ steps.  The conclusion is
that~$U$ runs in time exactly~$t_5:=r_0+t\cdot t_4+ q_0$ when~$M$
codes a suitable machine, where~$r_0$ is the time-point at which round
0 starts.  Note that~$r_0=r_0(M,\bar x,t,s,q)$ is the time of the
preparation phase. This is an explicit polynomial-time computation,
so~$r_0$ is~$\S^1_2$-provably bounded by a polynomial in~$n$.  Also
the time~$q_0$ of the wrap-up phase is provably bounded by a
polynomial in~$n$, and~$t_4$ is provably bounded by a polynomial
in~$s+n$.  Overall,~$t_5$ is~$\S^1_2$-provably bounded by~$t \cdot
(s+n)^c$ for some~$c \in \N$.

Additionally, we determine the exact time-points of the various phases of the
simulation. Round~$u < t$ starts at time~$r_u:=r_0+u\cdot t_4$.  For
all~$i < s_0$, it starts the simulation that fills cell~$i+1$ of the
configuration tape at time~$r_{u,i}:=r_u+1+i\cdot t_3$.  For
all~$j<|t_0|$, it starts the simulation of the~$j$-th step
of~$M_\start$ or~$M_\succ$ on~$i$ at time~$r_{u,i,j}:=r_{u,i}+j\cdot
t_1$. For all~$p<s_0$, the query-resolution computation for the~$j$-th
step of~$M_\start$ or~$M_\succ$ on~$i$ reads cell~$p+1$ of the tape
that holds~$X$ at time~$r_{u,i,j,p} := r_{u,i,j}+1+(2|s_0|+2)p$. The
wrap-up phase starts at time~$r_t = r_0+t\cdot t_4$.

We verify our time analysis in~$\S^1_2(\alpha)$.  Argue
in~$\S^1_2(\alpha)$ and assume for certain~$t',s',q'$, that~$Z$ is a
time-$t'$ space-$s'$ query-$q'$ computation of~$U^{\bar X}$
on~$(M,\bar x,t,s,q)$.  We want to show that~$\S^1_2(\alpha)$
proves~$\mathit{time}^Z(t')\leq t_5$.  Prove three facts about the
states of the clocks, the counters, the work tapes, and the
configuration tapes, as they appear in the configurations of~$Z$:~(1)
at time~$r_0$ within~$Z$, the two binary counters and the two clocks
are initialized at value~$0$, the rest of work tapes except the
configuration tapes appear clean up to
position~$\max\{|t_0|+1,|s_0|+1\}$, and all heads scan the~$0$
cell;~(2) for all~$u<t$, if at time~$r_u$ within~$Z$ the binary
clock~$0^{|t|}$ has value~$u$, the rest of tapes except the
configuration tapes are in the same state as at time~$r_0$, and all
heads scan cell~$0$, then at time~$u+1$ within~$Z$ the binary
clock~$0^{|t|}$ has value~$u+1$, the rest of tapes except the
configuration tapes are in the same state as at time~$r_0$, and all
heads scan cell~$0$; and~(3) if at time~$r_t$ within~$Z$ the binary
clock has value~$t$, the rest of tapes except the configuration tapes
are in the same state as at time~$r_0$, and all heads scan cell~$0$,
then at time~$t_5 = r_t + q_0$ within~$Z$ the computation is in a
halting state. If we succeed proving these three points, then the
claim follows: Apply~$\Delta^b_1(\alpha)$-induction to points~(1)
and~(2), with~(1) as base case, and~(2) as inductive case, to conclude
the premise of point~(3). Then apply this to point~(3) to get the
conclusion that~$\mathit{time}^Z(t') \leq t_5$.

We prove points (1)-(2)-(3). Points (1) and (3) follow from the fact
that the preparation phase, the update of the binary clock~$0^{|t|}$,
the comparison with the binary expansion of~$t$ from one of the input
tapes, and the wrap-up phase, are all done by explicit polynomial-time
machines with~$\S^1_2$-provable correctness and runtimes. Note also
for this that these machines are run on inputs and work tapes of
length polynomial in~$n$.  To prove point (2), assume~$u<t$ and that
at time~$r_u$ the binary clock~$0^{|t|}$ has value~$u$, the rest of
tapes except the configuration tapes are in the same state as at
time~$r_0$, and all heads scan cell~$0$.  

First prove that~$2|t|+2$ steps after~$r_u$ the counter has
value~$u+1$. Then prove, by $\Delta^b_1(\alpha)$-induction on~$i \leq s_0$,
for fixed~$X$ and~$u$ as parameters where~$X$ is the content of the
relevant configuration tape at time~$r_u$ within~$Z$, that at
time~$r_{u,i}+1$ within~$Z$ the value of the binary
counter~$0^{|s_0|}$ is~$i$ and the head of the other configuration
tape scans its~$i+1$ cell.  To prove this, first prove that~$2|s_0|+2$
steps after~$r_{u,i}+1$ the counter has value~$i+1$ and a copy of~$i$
is kept on a work tape. Then prove, by $\Delta^b_1(\alpha)$-induction on~$j
\leq |t_0|$, for fixed~$X,u,i$ as parameters, that at time~$r_{u,i,j}$
the simulation of~$M_\start$ or~$M_\succ$ on~$(M,\bar x,t,s,q,i)$ with
oracles~$\bar X,X$ is at its~$j$-th configuration, and that~$t_1$
steps later the head on the other configuration tape is moved one cell
right.  To prove this, prove by $\Delta^b_1(\alpha)$-induction on~$p \leq
s_0$, for fixed~$X,u,i,j$ as parameters, that at time~$r_{u,i,j,p}$
the head of the configuration tape that holds~$X$ scans cell~$p+1$,
and also by $\Delta^b_1(\alpha)$-induction on~$p \leq s_0+1$ that at
time~$r_{u,i,j,s_0}+p$ the same head scans cell~$s_0+1-p$. Finally
prove, by $\Delta^b_1(\alpha)$-induction on~$q\leq s_0$, that at
time~$r_{u,i,|t_0|}+q$ the head of the other configuration tape scans
its~$s_0+1-(q+1)$ cell. In particular, by time~$r_{u+1}$ both
configuration tapes have their heads at cell~$0$.  With this conclude
that at time~$r_{u+1}$ the value of the binary clock~$0^{|t|}$
is~$u+1$, all the work tapes except the configuration tapes are at the
state they were at time~$r_0$, and the heads of the configuration
tapes are at cell~$0$, as was to be proved to establish~(2).

For all of these proofs, use the fact that all the counter and clock
updates, copies, comparisons, and resets are implemented with explicit
polynomial-time machines with~$\S^1_2$-provable correctness
and~$\S^1_2$-provable polynomial runtime bounds. Also use the choice
of~$t_0$ and the fact that~$M_\start$ and~$M_\succ$ are explicit
polynomial-time machines and in particular~$\S^1_2$-provably terminate
in time at most~$|t_0|$.  One consequence of this is that the
configurations of the computations of these auxiliary machines can be
represented by numbers of polynomial in~$n$ many bits, so
$\Delta^b_1(\alpha)$-induction suffices to establish the uniqueness of
these computations and therefore that they agree with the ones that
appear within~$Z$.

\bigskip

\noindent\textbf{Space and query.}  The machine~$U$ uses tapes for the preparation, that is,
tapes for storing the clocks and their computations. These are
explicit~$\P$-machines, so their space is~$\S^1_2(\alpha)$-provably
bounded by a polynomial in~$n$. The same holds for the tapes used for
the simulations of~$M_\succ$ and~$M_\start$, and the counter and clock
updates, copies, comparisons, and resets.  The space used on the
configuration tapes is~$s_0+1$ which is bounded by a polynomial
in~$s+n$. Applying~$\Delta^b_1(\alpha)$-induction as above shows that,
at all times, either the flag bit is on or the position of the head
moving on a configuration tape is at most~$s_0+1$, namely one more
than the value of the binary counter~$0^{|s_0|}$. 
The flag bit is stored in the state, and the wrap-up phase does not
use any additional space.

To establish query bounds, observe that the machine~$U$ uses its query
tapes only to simulate~$M^{\bar X,X}_\succ$ on~$(M,\bar x,s,q,i)$, and
only to answer queries to~$\bar X$. By the check-steps 1-4 in the
definition of~$M_{\succ}$, the queries to~$\bar X$ that this machine
asks in step~5 have length at most~$|q|$,
and~$\Delta^b_1(\alpha)$-induction as above shows that, at all times,
either the flag bit is on, or the position of the head moving on a
query tape is less than~$|q|$.

\bigskip

\noindent\textbf{Proofs.}
The above shows (a). We prove (b) and (c). 

\medskip

To prove (c), argue
in~$\S^1_2(\alpha)$ and assume that~$Z$ is an accepting computation
of~$U^{\bar X}$ on~$\bar y$; recall~$\bar y = (M,\bar x,t,s,q)$. The
formula~$G$ is defined such that~$G(Z,\bar y,\cdot)$ is the sequence
of configurations on the relevant configuration tape at times~$r_u$
for~$u\le t$.  To verify that this is a partial time-$t$ space-$s$
query-$q$ computation of~$M^{\bar X}$ on~$\bar x$ we have to verify
that for every~$u<t$ and~$i<s_0$ each cell~$i+1$ of the relevant
configuration tape is filled with the bit indicating the truth value
of~$\Start(M,\bar x,s,q,i)$ if~$u=0$
or~$\Next(\bar X,X,M,\bar x,s,q,i)$ if~$u>0$, where~$X$ is the
configuration on the other configuration tape. We also need to verify,
for~$u>0$, that~$\Fail(\bar X,X,M,\bar x,s,q)$ does not hold.

The latter is easy: otherwise the flag bit would be set to~$1$ at
round~$u$ and, by~$\Delta^b_1(\alpha)$-induction, stay~$1$ until the
end, so~$Z$ would not be accepting.  For the former, we show for
all~$u,i$ that the configurations~$Y_{u,i,j}$ of the execution
of~$M_\start$ or~$M_\succ$ on~$i$ at times~$r_{u,i,j}$ within~$Z$,
with~$j\leq|t_0|$, form a computation, i.e., they are
the~$M_\start$-successors or~$M_\succ$-successors of one another.
Since these machines are polynomial-time, this is expressed by
a~$\Delta^b_1(\alpha)$-formula.  Hence,~$\Delta^b_1(\alpha)$-induction
suffices. We must also verify that the oracle queries are answered
correctly, i.e., that the right cell on the configuration tape
holding~$X$ is retrieved, and that the right bit is read. This is done
as in the time-analysis above: by~$\Delta^b_1(\alpha)$-induction
on~$p\leq s_0$, for fixed~$X,u,i,j$ as parameters, prove that at
time~$r_{u,i,j,p}$ within~$Z$ the head of the configuration tape that
holds~$X$ scans cell~$p+1$, the virtual oracle tape holds the same
value~$z$ as at time~$r_{u,i,j}$, and the tape that holds~$X$ has, at
cell~$z$, the same bit it had at time~$r_{u,i,j}$; in particular, 
this bit is
the bit scanned  at time~$r_{u,i,j,z}$.
%
%
%
This shows that~$G(Z,\bar y,\cdot)$ is a partial time-$t$ space-$s$
query-$q$ computation of~$M$ on~$\bar x$ with oracles~$\bar X$.
That~$F(G(Z,\bar y,\cdot),\bar y,\cdot)$ agrees with~$Z$ follows from
the definition of~$F$ below.

\medskip

To prove (b), argue in~$\S^1_2(\alpha)$ and let~$Y$ be a partial
time-$t$ space-$s$ query-$q$ computation of~$M^{\bar X}$ on~$\bar
x$. The formula~$F$ is defined such that~$F(Y,\bar y,\cdot)$ describes
a computation of~$U^{\bar X}$ on~$\bar y = (M,\bar x,t,s,q)$. We
describe this computation.  For~$u\le t$, the computation has the
configuration~$Y_u$ of~$Y$ written on the appropriate configuration
tape at time-point~$r_u$. For~$i<s_0$ and~$j \leq |t_0|$ it also has
the~$j$-th configuration~$Y_{u,i,j}$ of~$M_\start$
or~$M_\succ^{\bar X,Y_u}$ on~$(M,\bar x,s,q,i)$, at
time~$r_{u,i,j}$.  The configurations of the query-resolution
computations are also added between these: concretely, for
each~$p \leq s_0$ the configuration that has appropriate state and
agrees with the one at time~$r_{u,i,j}$ except that the binary
counter~$0^{|s_0|}$ has value~$p$ and the head of the relevant
configuration tape scans position~$p+1$ is added at
time~$r_{u,i,j,p}$. Additionally, it has the counter and clock
updates, copies, comparisons, and resets placed at the right time
intervals, as well as the polynomial-time preparation phase at the
beginning, and the polynomial-time wrap-up phase at the end.  The bits
of the thus described computation are computable in polynomial-time
with oracle~$Y$, and this is what the formula~$F$ does.

The assumption that~$Y$ is a computation implies for all~$u\leq t$
that~$Y_u=\Start(M,\bar x,s,q,\cdot)$ if~$u=0$
and~$Y_u=\Next(\bar X,Y_{u-1},M,\bar x,s,q,\cdot)$ if~$u>0$, and in
the latter case also that the
predicate~$\Fail(\bar X,Y_{u-1},M,\bar x,s,q)$ does not
hold. Since~$M_\start$ and~$M_\succ$ and the counter and clock
manipulations are polynomial-time machines
with~$\S^1_2(\alpha)$-provable correctness, and the query-resolution
computation is also~$\S^1_2(\alpha)$-provable correct as explained
above, the theory~$\S^1_2(\alpha)$ proves
that~$F(Y,\bar y,\cdot)$ is a computation of~$U^{\bar X}$
on~$\bar y = (M,\bar x,t,s,q)$.

\medskip

Finally, we argue that~$F(G(Z,\bar y,\cdot),\bar y,\cdot)$ agrees
with~$Z$ in~(c). We need that, for explicit polynomial-time machines,
any two computations on the same inputs and the same oracles are the
same. This is a special case of Lemma~\ref{lem:unique} for
polynomial-time machines, for which~$\S^1_2(\alpha)$ suffices. We also
need to prove that every query-resolution computation within~$Z$ is as
described by the formula~$F$. First prove
by~$\Delta^b_1(\alpha)$-induction that at every time~$p$, the head
positions and the content of the binary counter~$0^{|s_0|}$ are as
described by~$F$. For this it does not matter what the heads other
than the one on the binary counter are scanning. Once positions and
states are verified,~$\Delta^b_1(\alpha)$-induction verifies that
within~$Z$ no cell content is altered except on the binary counter.
\end{proof}

\subsection{Universal $\EXP$- and $\PSPACE$-machines}

Since all our applications of~$U$ below will concern explicit~$\EXP$-
or~$\PSPACE$-machines, without oracles, coded by standard~$M \in \N$,
the statement of Theorem~\ref{thm:univ} is more general than needed.
In this section we state the actual lemmas that we need and use as
black-boxes. If not stated otherwise, we understand an explicit
machine to be one with~$k=1$ input tapes and~$\ell=0$ oracle
tapes. Also, when referring to the universal machine~$U$, we mean the
one for~$k=1$ and~$\ell=0$. 

%
 
The machine $M_1$ of the first lemma
is used in our formalizations of $\EXP\not\subseteq\Ppoly$.

\begin{lemma}\label{lem:M1} 
  There is an explicit~$\EXP$-machine~$M_1$ such that for every
  explicit~$\EXP$-machine~$M$ witnessed, say, by the term~$t_M$, there
  are~$\Delta^b_1(\alpha)$-formulas~$F_1,G_1$ such that 
 \begin{enumerate}
 \item[(a)]  $\S^1_2(\alpha)\vdash$ \q{$Y$ is an accepting (rejecting) computation of $M$ on $x$}\\
\hspace*{8ex} $\ \to\q{$F_1(Y,x,\cdot)$ is an accepting (rejecting)  computation  of $M_1$ on $\langle M,x,t_M(x)\rangle$}$,
 \item[(b)] $\S^1_2(\alpha)\vdash$ \q{$Z$ is an accepting (rejecting) computation of $M_1$ on $\langle M,x,t_M(x)\rangle$}\\
\hspace*{8ex} $\ \to\q{$G_1(Z,x,\cdot)$ is an accepting (rejecting) computation  of $M$ on $x$}$.
 \end{enumerate}
\end{lemma}

\begin{proof} 
  The machine~$M_1$ on~$\langle N,x,t \rangle$ runs the universal
  machine~$U$ on~$(N,x,t,t,t)$; inputs not of this form are
  rejected. If~$U$ rejects, then so does~$M_1$. If~$U$ accepts, then~$M_1$
  accepts if and only if the computation of~$N$ on~$x$ simulated
  by~$U$ is accepting.

  Clearly,~$M_1$ can be defined as an explicit~$\EXP$-machine. To
  define~$F_1,G_1$,  fix first an
  explicit~$\EXP$-machine~$M$, witnessed by~$t_M$. Then we use~$F,G$
  from Theorem~\ref{thm:univ} to define
    $$
  \begin{array}{lll}
   F_1(Y,x,v) &:=&
  H_1\big(F\big(Y,M,x,t_M(x),t_M(x),t_M(x),\cdot\big),x,t_M(x),v\big), \\
   G_1(Z,x,v)& :=&
  G\big(H_0\big(Z,x,t_M(x),\cdot\big),M,x,t_M(x),t_M(x),t_M(x),v\big),
  \end{array}
  $$
  where~$H_1(W,x,t,v)$ is the formula which modifies a computation~$W$
  of~$U$ on~$(M,x,t,t,t)$ into a computation of~$M_1$
  on~$\langle M,x,t\rangle$ by adding the initial extraction
  of~$M,x,t$ from~$\langle M,x,t \rangle$ and adding the final check
  whether the simulated computation of~$M$ on~$x$ is accepting,
  and~$H_0(Z,x,t,v)$ is the formula which extracts from~$Z$ the
  computation of~$U$ on~$(M,x,t,t,t)$.

  We prove (a). Argue in~$\S^1_2(\alpha)$, assume~$Y$ is a halting
  computation of~$M$ on~$x$, and let~$t = t_M(x)$. By the proof of
  Theorem~\ref{thm:univ} (b), setting~$W := F(Y,M,x,t,t,t,\cdot)$ we
  get an accepting computation of~$U$ on~$(M,x,t,t,t)$ whose simulated
  computation of~$M$ on~$x$ is~$Y$. By the definition of~$M_1$ and the
  choice of~$F_1$, we conclude that~$F_1(Y,x,\cdot)$ is a halting
  computation of~$M_1$ on~$\langle M,x,t_M(x)\rangle$ which is
  accepting if and only if~$Y$ is accepting.

  We prove (b). Argue in~$\S^1_2(\alpha)$, assume~$Z$ is a halting
  computation of~$M_1$ on~$\langle M,x,t_M(x) \rangle$, and
  let~$t := t_M(x)$. Set~$W := H_0(Z,x,t,\cdot)$ to get a halting
  computation of~$U$ on~$(M,x,t,t,t)$.
  Set~$Y := G(W,M,x,t,t,t,\cdot)$.  If~$W$ is accepting, then by the
  proof of Theorem~\ref{thm:univ} (c) the set~$Y$ is a partial
  time-$t$ space-$t$ query-$t$ computation of~$M$ on~$x$ which agrees
  with the partial computation of~$M$ as found within~$W$.
  Since~$t = t_M(x)$, this partial computation is halting. By the
  definition of~$M_1$ and the choice of~$G_1$, we conclude
  that~$G_1(Z,x,\cdot)$ is a halting computation of~$M$ on~$x$ which
  is accepting if and only if~$Z$ is accepting.

  We are left to show that~$W$ is not rejecting.  Assume it is. Then
  either~$M$ does not code a suitable machine, which is false, or the
  flag bit within~$W$ is finally~$1$. Binary search gives a
  round~$u \leq t$ such that in round~$u$ the flag bit is updated
  from~$0$ to~$1$. Then~$u > 0$ and there is no~$u' < u$ such that
  round~$u'$ has flag bit~$1$: otherwise binary search would
  find~$u''$ such that~$u' < u'' < u$ where the flag bit changes
  from~$1$ to~$0$, contradicting the working of~$U$. Let~$V$ collect
  the configurations~$V_i$ with~$i < u$ on the configuration tapes
  within~$W$. Then~$V$ is a partial time-$u$ space-$t$ query-$t$
  computation of~$M$ on~$x$. Moreover, the flag bit
  indicates~$\Fail(V_{u-1},M,x,t,t)$. Since~$\ell=0$, this means that,
  according to the definition of~$M_{\succ}$, some head on some work
  tape would be off, i.e., it would visit some cell~$p \geq
  t$. In~$V_{u-1}$ this head's position is less than~$t$.  Hence this
  head has position~$t-1$ in~$V_{u-1}$. Since all positions are less
  than~$u$, this implies~$u=t$. Thus~$\mathit{time}^{V}(u) \geq u$ and
  the partial time-$u$ space-$t$ query-$t$ computation~$V$ contradicts
  the fact that~$t_M$ witnesses~$M$.
\end{proof}

Additionally, we have a formally verified machine~$M^*_1$ which, for each
explicit~$\EXP$-machine~$M$, computes like~$M^*$ from the
introduction. 

\begin{lemma}\label{lem:autoM1}
  There is an
  explicit~$\EXP$-machine~$M^*_1$ such that for every
  explicit~$\EXP$-machine~$M$ witnessed, say, by the term~$t_M$, there
  are~$\Delta^b_1(\alpha)$-formulas~$F_1^*,G_1^*$ such that
  \begin{enumerate} \itemsep=0pt
   \item[(a)]  $\S^1_2(\alpha)\vdash \q{$Y$ is a partial time-$t$ computation of $M$ on $x$} $\\
   \hspace*{8ex} $\  \wedge\ v\in Y \wedge v<\Bd(M,x,t,t_M(x),t_M(x))$ \\
   \hspace*{8ex} $\ \to\ \q{$F_1^*(Y,x,t,v,\cdot)$ is an accepting  computation  of $M^*_1$ on $\langle M, x,t,v,t_M(x)\rangle$}$,
 \item[(b)] $\S^1_2(\alpha)\vdash \q{$Z$ is an accepting computation of $M^*_1$ on $\langle M,x,t,v,t_M(x)\rangle$}$
 \\\hspace*{8ex}  $\ \to\ \q{$G_1^*(Z,x,t,\cdot)$ is a partial time-$t$ computation 
of $M$ on $x$}$ 
 \\\hspace*{8ex}  $\quad\quad  \wedge \ G_1^*(Z,x,t,v) \wedge  v<\Bd(M,x,t,t_M(x),t_M(x))$.
 \end{enumerate}
\end{lemma}

\begin{proof} 
  The machine~$M^*_1$ checks that its input has the
  form~$\langle N,x,t,v,s\rangle$ where~$N$ is a suitable machine
  and~$v=\langle u_v,i_v\rangle$ with~$u_v\le t$ and~$i_v<\Bd^0(N,x,s,s)$,
  and hence~$v<\Bd(N,x,t,s,s)$; inputs not of this form are
  rejected. If not rejected,~$M^*_1$ extracts~$x,u_v,i_v$ from its
  input and runs~$U$ on~$(N,x,t,s,s)$ with the following addition.  It
  uses an additional flag bit stored in its state, initially set
  to~$0$.  When~$U$ updates the clock holding the current value
  of~$u$, it checks whether~$u=u_v$. In case, and when updating the
  counter holding the current value of~$i<s_0$, it checks
  whether~$i=i_v$. In case, it updates the flag bit to the bit written
  into cell~$i_v+1$ of the configuration tape.  In the end,~$M_1^*$
  accepts if and only if~$U$ accepts and this flag bit is~$1$.

  Clearly,~$M_1^*$ is an explicit~$\EXP$-machine. To
  define~$F^*_1,G^*_1$ we  first let~$M$ be an
  explicit~$\EXP$-machine~$M$ witnessed by~$t_M$. Then we
  use~$F,G$ from Theorem~\ref{thm:univ} to define
  $$
  \begin{array}{lll}
     F^*_1(Y,x,t,v,w) &:=&
      H^*_1\big(F\big(Y,M,x,t_M(x),t_M(x),t_M(x),\cdot\big),x,t,v,t_M(x),w\big), \\
     G^*_1(Z,x,t,v)& :=&
      G\big(H_0\big(Z,x,t,t_M(x),\cdot\big),M,x,t_M(x),t_M(x),t_M(x),v\big),
  \end{array}
  $$
  where~$H^*_1(W,x,t,v,s,w)$ is the formula which modifies a
  computation~$W$ of~$U$ on~$(M,x,t,s,s)$ into a computation
  of~$M^*_1$ on~$\langle M,x,t,v,s\rangle$ by adding the initial
  extraction of~$M,x,t,v,s$ from the input, adding also suitable
  counter checks for managing the new flag bit, and adding the final
  acceptance check, and~$H_0(Z,x,t,s,v)$ extracts from~$Z$ the
  computation of~$U$ on~$(M,x,t,s,s)$.

  For (a), argue in~$\S^1_2(\alpha)$ and assume~$Y$ is a partial
  time-$t$ (space-$t_M(x)$ query-$t_M(x)$) computation of~$M$ on~$x$
  with~$v\in Y$ and~$v<\Bd(M,x,t,t_M(x),t_M(x))$.  Then the
  set~$W:=F(Y,M,x,t,t_M(x),t_M(x),\cdot)$ is an accepting computation
  of~$U$ on~$(M,x,t,t_M(x),t_M(x))$ and~$Z:=F_1^*(Y,x,t,v,\cdot)$ is a
  halting computation of~$M^*_1$ on~$\langle M,x,t,v,t_M(x)\rangle$.
  By the proof of Theorem~\ref{thm:univ}, the computation~$Y'$ of~$M$
  on~$x$ simulated within~$W$ agrees with~$Y$
  below~$\Bd(M,x,t,t_M(x),t_M(x))$.
  Since~$v<\Bd(M,x,t,t_M(x),t_M(x))$ and~$v\in Y$ we have~$v\in Y'$.
  This means that~$U$ in round~$u_v$ writes bit 1 into
  cell~$i_v+1$. Hence,~$M^*_1$ sets the flag bit to 1. This implies
  that~$Z$ is accepting.

  For (b), argue in~$\S^1_2(\alpha)$, and let~$Z$ be an accepting
  computation of~$M_1^*$ on~$\langle M,x,t,v,t_M(x)\rangle$. This
  implies~$v<\Bd(M,x,t,t_M(x),t_M(x))$.  Then~$Y:=G_1^*(Z,x,t,\cdot)$
  is the computation of~$M$ on~$x$ simulated by~$U$ within~$Z$.
  As~$Z$ is accepting, the flag bit is finally~$1$. Binary search
  finds the time point where the flag bit is changed from~$0$ to~$1$
  and, by the working of~$M^*_1$, this is happens when~$U$ is in
  round~$u_v$ filling cell~$i_v+1$. Hence this cell gets bit~$1$,
  meaning~$v\in Y$, that is,~$G_1^*(Z,x,t,v)$.
\end{proof}

We state analogous lemmas for~$\PSPACE$.


\begin{lemma}\label{lem:M2} 
  There is an explicit~$\PSPACE$-machine~$M_2$ such that for every
  explicit~$\PSPACE$-machine~$M$ witnessed, say, by the term~$t_M$, there
  are~$\Delta^b_1(\alpha)$-formulas~$F_2,G_2$ such that 
 \begin{enumerate}
 \item[(a)]  $\S^1_2(\alpha)\vdash$ \q{$Y$ is an accepting (rejecting) computation of $M$ on $x$}\\
\hspace*{8ex} $\ \to\q{$F_2(Y,x,\cdot)$ is an accepting (rejecting)  computation  of $M_2$ on $\langle M,x,t_M(x)\rangle$}$,
 \item[(b)] $\S^1_2(\alpha)\vdash$ \q{$Z$ is an accepting (rejecting) computation of $M_2$ on $\langle M,x,t_M(x)\rangle$}\\
\hspace*{8ex} $\ \to\q{$G_2(Z,x,\cdot)$ is an accepting (rejecting) computation  of $M$ on $x$}$.
 \end{enumerate}
\end{lemma}

\begin{proof} 
Machine~$M_2$ is defined like~$M_1$ but runs the universal machine~$U$
on~$(N,x,t,|t|,t)$ instead of running it on~$(N,x,t,t,t)$.
 \end{proof}

\begin{lemma}\label{lem:autoM2}
  There is an
  explicit~$\PSPACE$-machine~$M^*_2$ such that for every
  explicit~$\PSPACE$-machine~$M$ witnessed, say, by the term~$t_M$,
  there are~$\Delta^b_1(\alpha)$-formulas~$F_2^*,G_2^*$ such that
  \begin{enumerate} \itemsep=0pt
   \item[(a)]  $\S^1_2(\alpha)\vdash \q{$Y$ is a partial time-$t$ computation of $M$ on $x$} $\\
   \hspace*{8ex} $\  \wedge\ v\in Y \wedge v<\Bd(M,x,t,|t_M(x)|,t_M(x))$ \\
   \hspace*{8ex} $\ \to\ \q{$F_2^*(Y,x,t,v,\cdot)$ is an accepting  computation  of $M^*_2$ on $\langle M, x,t,v,t_M(x)\rangle$}$,
 \item[(b)] $\S^1_2(\alpha)\vdash \q{$Z$ is an accepting computation of $M^*_2$ on $\langle M,x,t,v,t_M(x)\rangle$}$
 \\\hspace*{8ex}  $\ \to\ \q{$G_2^*(Z,x,t,\cdot)$ is a partial time-$t$ computation 
of $M$ on $x$}$ 
 \\\hspace*{8ex}  $\quad\quad  \wedge \ G_2^*(Z,x,t,v) \wedge  v<\Bd(M,x,t,|t_M(x)|,t_M(x))$.
 \end{enumerate}
\end{lemma}

\begin{proof} 
  Like Lemma~\ref{lem:autoM1} except that~$M_2^*$
  on~$\langle N,x,t,v,s\rangle$ runs~$U$ on~$(N,x,t,|s|,s)$ with the
  same addition, instead of running it on~$(N,x,t,s,s)$.  Also, the
  bound~$\Bd^0(N,x,s,s)$ used in the initial phase
  becomes~$\Bd^0(N,x,|s|,s)$.
\end{proof}

For the rest of this work we fix {\em any} machines~$M_1$ and~$M_2$
satisfying Lemmas~\ref{lem:M1} and~\ref{lem:M2}.

\section{Formalizations of circuit lower bounds}\label{sec:formalizations}

The $\alpha$- and $\meyer$-formulas from the introduction are defined as follows:

\begin{definition}
Let $c\in\N$ and let $M$ be an explicit $\EXP$- or $\PSPACE$-machine.
\begin{eqnarray*}
\alpha^c_M&:=& \forall n{\in}\Log\  \exists C{<}2^{n^c}\ \forall x{<}2^n\\
&&\quad \big(C( x)=1\ \leftrightarrow \ \exists_2 Y \ \q{$Y$ is an accepting  computation of $M$ on $x$}\big),\\
\meyer^c_M&:=&  \forall n{\in}\Log\  \exists C{<}2^{n^c}\ \forall x{<}2^n\\
&&\quad  \q{$C_{x}(\cdot)$ is a halting computation of $M$ on $x$}.
\end{eqnarray*}
\end{definition}

Both theories~$\{\neg\alpha^c_{M_1}\mid c\in\N\}$
and~$\{\neg\meyer^c_{M_1}\mid c\in\N\}$ are true (in the standard
model) if and only if~$\EXP\not\subseteq\Ppoly$. The same holds
for~$M_2$ and~$\PSPACE$, instead of~$M_1$ and~$\EXP$.  The aim of this
section is two-fold:~(1)~verify over weak theories that these
formalizations do not depend on the particular choices of the
machines~$M_1$ and~$M_2$ in Lemmas~\ref{lem:M1} and~\ref{lem:M2},
and~(2)~verify over weak theories that both formalizations are
equivalent; i.e.:

\begin{theorem} \label{thm:equiv}
For every explicit $\EXP$-machine $M$ and 
every~$c \in \N$ there is~$d \in \N$ such that
\medskip

\begin{tabular}{llllllllllllll}
(a) & $\S^1_2(\alpha) \vdash (\alpha^c_{M_1} \to \alpha^d_M),$ & \hspace{0.5cm} &
(c) & $\T^1_2(\alpha) \vdash (\meyer^c_{M_1} \to \alpha^d_{M_1}),$ \\
(b) & $\S^1_2(\alpha) \vdash (\meyer^c_{M_1} \to \meyer^d_M)$, & &
(d) & $\T^1_2(\alpha) \vdash (\alpha^c_{M_1} \to \meyer^d_{M_1})$.
\end{tabular}
\end{theorem}

\noindent Theorem~\ref{thm:equiv2} below states the analogue for $\PSPACE$.
Following the principle that the logically simplest formulas are
preferrable in definitions, we take:

%
 
\begin{definition} $\q{$\EXP\not\subseteq\Ppoly$} := \big\{\neg\meyer^c_{M_1}\mid c\in\N\big\} $ and $
 \q{$\PSPACE\not\subseteq\Ppoly$} := \big\{\neg\meyer^c_{M_2}\mid c\in\N\big\}$.
\end{definition}

We start the proof of Theorem~\ref{thm:equiv} with the proof that the
definition does not depend on the particular choice of~$M_1$.
 
\begin{lemma}\label{lem:alpha} 
For every explicit~$\EXP$-machine~$M$ and every~$c\in\N$ there
is~$d\in\N$ such that
$$
\S^1_2(\alpha) \ \vdash \ (\alpha_{M_1}^c\to\alpha_{M}^d)\  \wedge\  (\meyer_{M_1}^c\to\meyer_{M}^d).
$$ 

\end{lemma}

\begin{proof} 
We prove each implication for a separate~$d\in\N$; the maximum of the
2 values of~$d$ then witnesses the lemma.

For the first implication, Lemma~\ref{lem:M1}
and~$\Delta^b_1(\alpha)$-comprehension imply that~$\S^1_2(\alpha)$
proves
\begin{equation*}
\begin{split}
& \exists_2 Y\  \q{$Y$ is an accepting  computation of $M$ on $x$}\\
&  \leftrightarrow\ \exists_2 Y\  \q{$Y$ is an accepting  computation of $M_1$ on $\langle M, x,t_M( x)\rangle$}.
\end{split}
\end{equation*}
Argue in~$\S^1_2(\alpha)$ and
assume~$\alpha_{M_1}^c$. Given~$n\in\Log$ choose~$e\in\N$ and
set~$m:=n^e$ so that~$\langle M,x,t_M(x)\rangle<2^{m}$ for
all~$x<2^n$. Apply~$\alpha^c_{M_1}$ for this~$m$ and
get~$C'<2^{m^c}=2^{n^{ec}}$ such that for all~$x<2^n$:
\begin{equation*}
\begin{split}
& C'(\langle M, x,t_M( x)\rangle)=1 \\ 
& \leftrightarrow \exists_2 Y\  \q{$Y$ is an accepting  computation of $M_1$ on $\langle M, x,t_M( x)\rangle$}.
\end{split}
\end{equation*}
By Lemma~\ref{lem:circuit} there are small circuits computing~$\langle
M,x, t_M( x)\rangle$ from~$x$. This gives a circuit~$C$ such
that~$C(x)= C'(\langle M,x, t_M( x)\rangle)$ for
all~$x<2^n$. Clearly,~$C<2^{n^d}$ for suitable~$d\in\N$. This~$C$
witnesses~$\alpha_M^d$ for~$n\in\Log$.

For the second implication, argue in~$\S^1_2(\alpha)$ and
assume~$\meyer_{M_1}^c$.  Given~$n\in\Log$ choose~$e\in\N$ and
set~$m:=n^e$ so that~$\langle M,x,t_{M}(x)\rangle<2^{m}$ for
all~$x<2^n$. Apply~$\meyer_{M_1}$ for this~$m$ and
get~$C'<2^{m^c}=2^{n^{ec}}$ such that for all~$x<2^n$:
\begin{equation*}
 \q{$C'_{\langle M,x,t_{M}( x)\rangle}(\cdot)$ is a halting computation of $M_1$ on $\langle  M, x,t_{M}( x)\rangle$}.
\end{equation*}
Using the formula~$G_1$ from Lemma~\ref{lem:M1}~(b), we have, for
all~$x<2^n$: 
$$
\q{$G_1\big(C'_{\langle M,x,t_{M}(x)\rangle}(\cdot), x,\cdot \big)$ is a halting computation of $M$ on $x$}.
$$
Choose $f\in\N$ so that $\Bd_M(x)<2^{n^f}$ for all $x<2^n$.
By Lemma~\ref{lem:circuit} there is a circuit $C$ such that  for all $x<2^n$ and all $v<2^{n^f}$: 
$$
C(x,v)=1\leftrightarrow G_1\big(C'_{\langle M, x,
  t_{M}(x)\rangle}(\cdot),x,v\big).
$$
Choose $d\in\N$ such that $C<2^{n^d}$. Then $C$ witnesses $\meyer_{M}^d$ for $n$.
\end{proof}

The next lemma states one half of the equivalence of the two
formalizations.  The proof of this is still straightforward but we
need the theory~$\T^1_2(\alpha)$ instead of~$\S^1_2(\alpha)$ due to an
application of Lemma~\ref{lem:unique}.

\begin{lemma}\label{lem:gammaalpha} 
For every explicit $\EXP$-machine $M$ and 
every $c\in\N$ there is
$d\in\N$ such that 
$$
\T^1_2(\alpha)\vdash( \meyer^c_M\to\alpha^c_{M}).
$$
\end{lemma}
\begin{proof} 
Argue in~$\T^1_2(\alpha)$ and assume~$\meyer^c_M$. Let~$n\in\Log$ and
choose~$C'$ according to~$\meyer_M^c$ for~$n$.  Let~$A(Y,x)$
be
a~$\Delta^b_1(\alpha)$-formula which states that~$Y$ is accepting: it
applies a~$\PV(\alpha)$-function to retrieve the state of
the~$t_M(x)$-th configuration in~$Y$ and checks that it is an
accepting state of~$M$. The formula does not check that~$Y$ is a
halting computation of~$M$ on~$x$; it assumes that it is and may
behave unexpectedly otherwise.

Lemma~\ref{lem:circuit} gives a~$d\in\N$ and a circuit~$C<2^{n^d}$
such that for all~$x<2^n$:
$$
C(x)=1\leftrightarrow A(C'_x(\cdot),x).
$$
We claim that~$C$ witnesses~$\meyer_M^d$ for~$n$. If~$C(x)=1$,
then~$C'_x(\cdot)$ is an accepting computation of~$M$
on~$x$. If~$C(x)=0$, then~$C'_x(\cdot)$ is not an accepting
computation of~$M$ on~$x$. Since it is a halting computation, it is
rejecting. It exists by~$\Delta_1(\alpha)$-comprehension. By
Lemma~\ref{lem:unique} (here we use~$\T^1_2(\alpha)$), there is no
accepting computation~$Y$ of~$M$ on~$x$.
\end{proof}

We are left to prove the second half of the equivalence of the two
formalizations, i.e., to infer~$\meyer$ from~$\alpha$. This is Meyer's
argument and formalizing it is not trivial. We do it in two steps
passing through an auxiliary formula~$\bar\beta$, which is a variant
of the~$\Pi^{1,b}_1$-formula in the~$\beta$-formalization
from~\cite{abm}. First, we infer~$\bar\beta$ from~$\alpha$, and
then~$\meyer$ from~$\bar \beta$.

An intuitive reason for something
non-trivial going on is that, unlike~$\alpha$, the
formula~$\bar\beta$ is universal in the set sort. In
particular,~$\bar\beta$ does not seem to imply the existence of
exponential-time computations. In contrast,~$\meyer$ straightforwardly
implies the existence of such computations
via~$\Delta^b_1(\alpha)$-comprehension. Again we need~$\T^1_2(\alpha)$
instead of~$\S^1_2(\alpha)$ due to an application of
Lemma~\ref{lem:unique} in the first step, and a key application
of~$\Pi^b_1$-induction in the second step.

\begin{lemma}\label{lem:alphagamma} 
For every explicit $\EXP$-machine $M$ and every $c\in\N$ there is
$d\in\N$ such that
$$
\T^1_2(\alpha)\vdash (\alpha_{M_1}^c\to\meyer_M^d).
$$ 
\end{lemma}

\begin{proof} 
For every explicit~$\EXP$-machine~$M$ witnessed, say, by
  term~$t_M$, and every~$c\in\N$ set
\begin{eqnarray*}
\bar\beta^c_{M}&:=&\forall n{\in}\Log\ \exists C{<}2^{n^c}\ \forall
x{<}2^n\ \forall t{\leq}t_M(x)\ \forall_2 Y\\ &&\quad \big(\q{$Y$ is a
  partial time-$t$ computation of~$M$ on~$ x$} \\[-1ex] &&\quad\quad
\ \to \ \forall v{<}\Bd(M,x,t,t_M(x),t_M(x))\ (C(x,t,v) = 1 \leftrightarrow 
v \in Y)\big).
\end{eqnarray*}

The lemma follows from the following two steps.\bigskip

\noindent \textbf{First step.}  For every explicit~$\EXP$-machine~$M$ and
every~$c\in\N$ there is~$d\in\N$ such that
$$
\T^1_2(\alpha)\vdash(\alpha^c_{M_1}\to\bar\beta^d_{M}).
$$ 
Let $t_M$ witness $M$. Argue in~$\T^1_2(\alpha)$ and assume~$\alpha^c_{M_1}$.  Let~$M^*_1$
be a machine
according to Lemma~\ref{lem:autoM1}. This lemma
and~$\Delta^b_1(\alpha)$-comprehension imply for $v<\Bd(M,x,t,t_M(x),t_M(x))$:
\begin{eqnarray*}
&&\exists_2 Y\ \big( \q{$Y$ is a partial time-$t$ 
computation of $M$ on $x$} \wedge v\in Y\big)\\
&&\ \leftrightarrow \ \exists_2 Z\ \q{$Z$ is an accepting computation of 
$M^*_1$ on $\langle M,x,t,v,t_M(x)\rangle$}.
\end{eqnarray*}
By Lemma~\ref{lem:alpha} we have~$\alpha^{e}_{M^*_1}$ for
some~$e\in\N$.  Let~$n\in\Log$. Choose~$f\in\N$ such for all~$x<2^n$
and~$t\le t_M(x)$, and every~$v<\Bd(M,x,t,t_M(x),t_M(x))$ we
have~$\langle M,x,t,v,t_M(x)\rangle<2^{n^f}$.  Set~$m:=n^{f}$.
Choose~$C'<2^{m^{e}}=2^{n^{fe}}$ according to~$\alpha_{M^*_1}^e$
for~$m$. By Lemma~\ref{lem:circuit}, there are~$d\in\N$ and a
circuit~$C<2^{n^d}$ such
that~$C(x,t,v)=C'(\langle M,x,t,v,t_M(x)\rangle)$ for all~$x,t,v$ as
above. Then, for all such $x,t,v$:
$$
C(x,t,v)=1\ \leftrightarrow\ \exists_2 Y\ (\q{$Y$ is a partial time-$t$
computation of $M$ on $x$}\wedge v\in Y ).
$$
To see that~$C$ witnesses~$\bar\beta^d_M$ for~$n$, assume~$Y$ is a
partial time-$t$ computation of~$M$ on~$x$. We claim
$$ 
\forall v{<}\Bd(M,x,t,t_M(x),t_M(x))\ (C(x,t,v)=1 \leftrightarrow v \in Y).
$$ 
Let~$v < \Bd(M,x,t,t_M(x),t_M(x))$.  Clearly,~$v \in Y$
implies~$C(x,t,v)=1$. Conversely,~$C(x,t,v)=1$ implies~$v\in Y'$ for
some partial time-$t$ computation~$Y'$ of~$M$ on~$x$.  By
Lemma~\ref{lem:unique} (here we use~$\T^1_2(\alpha)$), the sets~$Y$
and~$Y'$ agree below~$\Bd(M,x,t,t_M(x),t_M(x))$. Therefore,~$v \in Y$.

\bigskip

\noindent \textbf{Second step.} For 
every explicit $\EXP$-machine $M$ and every $c \in \N$ there is $d\in\N$
such that
$$
\T^1_2(\alpha)\vdash(\bar\beta^c_{M}\to\meyer^d_{M}).
$$ 
Let~$t_{M}$ witness~$M$. Argue in~$\T^1_2(\alpha)$ and
assume~$\bar\beta^c_{M}$. 
Let~$n\in\Log$ and 
choose~$D<2^{n^c}$ according to~$\bar\beta_{M}$ for
this~$n$. 

\medskip

\noindent{\em Claim:} 
$x<2^n \wedge t\le t_M(x) \to \q{$D(x,t,\cdot)$ is a partial
  time-$t$ computation of $M$ on $x$}$.

\medskip

The claim implies the lemma: by Lemma~\ref{lem:circuit} there is
a~$d \in \N$ and a circuit~$C < 2^{n^d}$ such
that~$C(x,v) = D(x,t_M(x),v)$ for~$x<2^n$
and~$v<\Bd(M,x,t_M(x),t_M(x),t_M(x))=\Bd_M(x)$; thus,~$C$
witnesses~$\meyer^d_M$ for~$n$.

We are left to prove the claim. Note that its statement is
a~$\Pi^b_1$-formula. We prove it by induction on~$t$ (this
requires~$\T^1_2(\alpha)$).

For~$t=0$, the set~$Z:=\Start(M,x,t_M(x),t_M(x),\cdot)$ is a partial
time-$0$ computation of~$M$ on~$x$ consisting only of the start
configuration of~$M$ on~$x$.  Note that~$Z$ exists
by~$\Delta^b_1(\alpha)$-comprehension.  By~$\bar\beta^c_{M}$, the
set~$D(x,0,\cdot)$ equals~$Z$ below~$\Bd(M,x,0,t_M(x),t_M(x))$ and
therefore it is a partial time-$0$ computation of~$M$ on~$x$.

Assuming the claim for~$t<t_M(x)$ we verify it
for~$t+1$. Let~$Z:=D(x,t,\cdot)$. By the induction hypothesis the
set~$Z$ is a partial time-$t$ computation of~$M$ on~$x$. Note that~$Z$
exists by~$\Delta^b_1(\alpha)$-comprehension. Let~$Z'$ be the result
of adding one computation step to~$Z$; i.e., the
configuration~$Z_{t+1} := \Next(Z_t,M,x,t_M(x),t_M(x),\cdot)$ is
appended to~$Z$, where~$Z_t$ is the last configuration of~$Z$. Observe
that~$\Fail(Z_t,M,x,t_M(x),t_M(x))$ 
does not hold because the term~$t_M(x)$ witnesses~$M$
and~$t<t_M(x)$. The resulting~$Z'$ is then a partial time-$(t+1)$
computation~$Z'$ of~$M$ on~$x$. 
Note~$Z'$ exists
by~$\Delta^b_1(\alpha)$-comprehension. Now,~$\bar\beta^c_{M}$ implies
that~$D(x,t+1,\cdot)$ equals~$Z'$ below~$\Bd(M,x,t+1,t_M(x),t_M(x))$ and
therefore it is a partial time-$(t+1)$-computation of~$M$ on~$x$.
%
\end{proof}

\begin{proof}[Proof of Theorem~\ref{thm:equiv}.] 
Parts (a) and (b) are Lemma~\ref{lem:alpha}. Parts (c) and (d) are the
special cases of Lemma~\ref{lem:gammaalpha} and
Lemma~\ref{lem:alphagamma} when~$M = M_1$.
\end{proof}

In the analogue of Theorem~\ref{thm:equiv} for~$\PSPACE$ only
the last implication seems to need~$\T^1_2(\alpha)$:

\begin{theorem} \label{thm:equiv2}
For every explicit $\PSPACE$-machine $M$ and 
every~$c \in \N$ there is~$d \in \N$ such that
\medskip

\begin{tabular}{llllllllllllll}
(a) & $\S^1_2(\alpha) \vdash (\alpha^c_{M_2} \to \alpha^d_M),$ & \hspace{0.5cm} &
(c) & $\S^1_2(\alpha) \vdash (\meyer^c_{M_2} \to \alpha^d_{M_2})$, \\
(b) & $\S^1_2(\alpha) \vdash (\meyer^c_{M_2} \to \meyer^d_M),$ & & 
(d) & $\T^1_2(\alpha) \vdash (\alpha^c_{M_2} \to \meyer^d_{M_2})$.
\end{tabular}
\end{theorem}

\begin{proof}
  Parts (a) and (b) are proved as in Lemma~\ref{lem:alpha} but for
  explicit~$\PSPACE$-machines~$M$ and~$M_2$ (instead of~$M_1$). The
  proofs use Lemma~\ref{lem:M2} instead of Lemma~\ref{lem:M1}. Part
  (c) is the analogue of Lemma~\ref{lem:gammaalpha} and is proved in
  the same way. That proof required~$\T^1_2(\alpha)$ only in the final
  use of Lemma~\ref{lem:unique}. As stated in Lemma~\ref{lem:unique}
  itself, this works in~$\S^1_2(\alpha)$ for
  explicit~$\PSPACE$-machines. Part (d) is the analogue of
  Lemma~\ref{lem:alphagamma} and is proved in the same way but we
  use~$\Bd(M,x,t,|t_M(x)|,t_M(x))$ instead
  of~$\Bd(M,x,t,t_M(x),t_M(x))$.  That proof required~$\T^1_2(\alpha)$
  in two places: in the application of Lemma~\ref{lem:unique} and in
  the proof of the claim. As in Part (c), the first can be replaced
  by~$\S^1_2(\alpha)$ but the second still needs~$\Pi^b_1$-induction
  and hence~$\T^1_2$.
\end{proof}


%
%

\section{Consistency }\label{sec:cons}

This  section has three subsections. One for the proof of Theorem~\ref{thm:maincon}~(a), one for
Theorem~\ref{thm:maincon}~(b) and one for Theorem~\ref{thm:nbsort}.

\subsection{Consistency of the $\meyer$-formalization}

Theorem~\ref{thm:maincon}~(a) follows from Takeuti's Theorem~\ref{thm:takeuti} and the following.

\begin{theorem}\label{thm:congammaM}
  Assume~$(N,\mathcal Y)\models\S^1_2(\alpha)+\neg\forall\bar
  x\varphi(\bar x)$ where~$\varphi(\bar x)$ is
  a~$\Sigma^{1,b}_1$-formula without free set variables.
\begin{enumerate}\itemsep=0pt
\item[(a)] If $\V^1_2\vdash \varphi(\bar x)$, then $N\models\q{$\EXP\not\subseteq\Ppoly$}.$
\item[(a)] If $\U^1_2\vdash \varphi(\bar x)$, then
$N\models\q{$\PSPACE\not\subseteq\Ppoly$}.$
\end{enumerate}
\end{theorem}

As outlined in the introduction, the proof  relies on ``new-style'' witnessing theorems from~\cite{bb} (improving \cite{knt}):

\begin{theorem}\label{thm:wit} 
Let $\varphi(\bar X,\bar x)$ be a $\hat\Sigma^{1,b}_1$-formula.
\begin{enumerate}\itemsep=0pt
\item[(a)] If $\V^1_2\vdash \varphi(\bar X,\bar x)$, then there is an explicit oracle
$\EXP$-machine $M$ such that
$$
\S^1_2(\alpha)\vdash \q{$Y$ is a halting computation of $M^{\bar X}$ on $\bar x$} \to \varphi(\bar X,\bar x).
$$
\item[(b)] If $\U^1_2\vdash \varphi(\bar X,\bar x)$, then there is an explicit oracle
$\PSPACE$-machine $M$ such that
$$
\S^1_2(\alpha)\vdash \q{$Y$ is a halting computation of $M^{\bar X}$ on $\bar x$} \to \varphi(\bar X,\bar x).
$$ 
\end{enumerate}
\end{theorem}

\begin{proof} We comment on how to transfer the formalism
  from~\cite{bb} to ours. The statements in~(a) and~(b) follow
  from~\cite[Theorem~4.8]{bb} and~\cite[Theorem~4.6]{bb}. The
  conclusions of~(a) and~(b) are weaker than the conclusions stated
  there -- see \cite[Theorem~4.3]{bb}. The definitions
  of~$\U^1_2,\V^1_2$ in~\cite{bb} slightly differ from ours but are
  easily checked to have the same~$\hat\Sigma^{1,b}_1$ consequences.
  Our notions of {\em explicit}~$\EXP$- or~$\PSPACE$-machines can be
  used as a precise version of the more loosely used terms {\em
    explicitly} polynomial-space or exponential-time machines
  in~\cite{bb}.
\end{proof}

Note that every~$\meyer$-formula implies the existence of computations
by~$\Delta^b_1(\alpha)$-comprehension, but only for machines without
oracles. Intuitively, a~$\meyer$-formula enables the capacity of a
weak theory to simulate a strong one, for consequences without set
variables. The following makes this precise.

\begin{corollary} \label{cor:Sigma11cons} Let $c\in\N$ and let $\varphi(\bar x)$ be a $\Sigma^{1,b}_1$-formula without free set variables.
\begin{enumerate}\itemsep=0pt
\item[(a)] If $\V^1_2\vdash\varphi(\bar x)$, then $\S^1_2(\alpha)+\meyer^c_{M_1}\vdash\varphi(\bar x)$.
\item[(b)] If $\U^1_2\vdash\varphi(\bar x)$, then $\S^1_2(\alpha)+\meyer^c_{M_2}\vdash\varphi(\bar x)$.
\end{enumerate}
\end{corollary}

\begin{proof} We prove only~(b); the proof of~(a) is analogous. We
  treat strict formulas first. Assume
  that~$\varphi(\bar x)\in\hat\Sigma^{1,b}_1$
  and~$\U^1_2\vdash\varphi(\bar x)$. Say,~$\bar x=(x_1,\ldots,
  x_k)$. Let~$\varphi'(x):=\varphi(\pi_1(x),\ldots,\pi_k(x))$ where
  the~$\pi_i\in\PV$ compute the
  projections~$\pi_i(\langle x_1,\ldots, x_k\rangle)=x_i$.
  Then~$\S^1_2(\alpha)\vdash(\forall
  x\varphi'(x)\leftrightarrow\forall\bar x\varphi(\bar x))$. In
  particular,~$\U^1_2\vdash\varphi'(x)$. Choose~$M$
  for~$\varphi'(x)\in\hat\Sigma^{1,b}_1$ according to
  Theorem~\ref{thm:wit}~(b). Note~$M$ has one input tape and no
  oracles.  By
 Theorem~\ref{thm:equiv2}~(b),~$\S^1_2(\alpha)+\meyer^c_{M_2}\vdash\meyer_{M}^d$
  for some~$d\in\N$.  We
  claim~$\S^1_2(\alpha)+\meyer^d_{M}\vdash \varphi'(x)$.

  Argue in~$\S^1_2(\alpha)+\meyer^d_{M}$ and let~$x$ be
  given. Set~$n:=\max\{|x|,2\}$ and choose~$C$ according
  to~$\meyer_{M}^d$ for~$n$.  Then~$x<2^n$ and~$Y:=C_x(\cdot)$ is a
  halting computation of~$M$ on~$x$.  Note~$Y$ exists
  by~$\Delta^b_1(\alpha)$-comprehension. Then get~$\varphi'(x)$ by
  Theorem~\ref{thm:wit}~(b).

Non-strict formulas are handled with Lemma~\ref{lem:strict}:
 assume~$\varphi(\bar  x)\in\Sigma^{1,b}_1$ and~$\U^1_2\vdash\varphi(\bar x)$; 
 then~$\U^1_2\vdash\hat\varphi(\bar x)$ and $\hat\varphi\in\hat\Sigma^{1,b}_1$; hence,~$\S^1_2(\alpha)+\meyer^c_{M_2}\vdash\hat\varphi(\bar x)$, so~$\S^1_2(\alpha)+\meyer^c_{M_2}\vdash\varphi(\bar x)$.
%
 \end{proof}
 
 \begin{remark} We do not need the full strength of
   Theorem~\ref{thm:wit}~(a): we can replace~$\S^1_2(\alpha)$
   by~$\U^1_2$ and then argue for Corollary~\ref{cor:Sigma11cons}~(a)
   as follows. From~$\V^1_2\vdash\varphi(\bar x)$
   infer~$\U^1_2\vdash(\meyer^c_{M_1}\to\varphi(\bar x))$ as seen. But
   this is logically equivalent to a~$\Sigma^{1,b}_1$ formula, still
   without free set variables. By
   Lemma~\ref{lem:alpha},~$\S^1_2(\alpha)+\meyer^c_{M_1}\vdash\meyer_{M_2}^d$
   for some~$d\in\N$.  By
   Corollary~\ref{cor:Sigma11cons}~(b),~$\S^1_2(\alpha)+\meyer^c_{M_2}\vdash(\meyer^c_{M_1}\to\varphi(\bar
   x))$. Hence~$\S^1_2(\alpha)+\meyer^c_{M_1}\vdash \varphi(\bar x)$.
  \end{remark}

\begin{proof}[Proof of Theorem~\ref{thm:congammaM}] 
  We prove only (a); the proof of (b) is analogous.
  If~$\varphi(\bar x)$ and~$(N,\mathcal Y)$ are as stated, (a) of the
  corollary implies~$(N,\mathcal Y)\models\neg\meyer^c_{M_1}$ for
  all~$c\in\N$.
\end{proof}


\subsection{Consistency of the $\alpha$-formalization}

While we do not know whether Lemma~\ref{lem:alphagamma} holds
for~$\S^1_2(\alpha)$, we can show a weakening:

\begin{lemma}\label{lem:alphagammaproof}\
For every $c \in \N$ there is $d \in \N$ such that
\begin{enumerate}\itemsep=0pt
\item[(a)] if $\S^1_2(\alpha)\vdash\alpha^c_{M_1}$, then $\S^1_2\vdash\meyer^d_{M_1}$,
\item[(b)]  if $\S^1_2(\alpha)\vdash\alpha^c_{M_2}$, then $\S^1_2\vdash\meyer^d_{M_2}$.
\end{enumerate}
\end{lemma}

Before we prove it, we argue that this implies our main consistency result:

\begin{proof}[Proof of Theorem~\ref{thm:maincon}.] As already stated,
  Part~(a) follows from Theorems~\ref{thm:takeuti}
  and~\ref{thm:congammaM}. For Part~(b),
  assume~$\S^1_2(\alpha)\cup\{\neg\alpha^c_{M_1}\mid c\in\N\}$ is
  inconsistent. By compactness,~$\S^1_2(\alpha)\vdash \alpha^c_{M_1}$
  for some~$c\in\N$. By Part~(a) of the
  lemma,~$\S^1_2\vdash \meyer^d_{M_1}$ for some~$d\in\N$. This
  contradicts Part~(a) of the theorem, which we just proved.
\end{proof}

The proof of Lemma~\ref{lem:alphagammaproof} is by analysis of the
models~$(N,\C_N)$ from Lemma~\ref{lem:CM} and relies on formally
verified model-checkers; concretely, the proof uses the following weak
version of~\cite[Lemma~20]{abm}:

\begin{lemma}\label{lem:mc} 
  For every~$\Pi^{b}_1$-formula~$\varphi(\bar x)$ there is an
  explicit~$\PSPACE$-machine~$M_\varphi$ and
  a~$\Delta^b_1(\alpha)$-formula~$C_\varphi(\bar x,v)$ such that
\begin{enumerate}\itemsep=0pt
\item[(a)] $\S^1_2(\alpha)\vdash \q{$Y$ is an accepting computation of $M_\varphi$ on $\bar x$}
\ \to\ \varphi(\bar x)$,
\item[(b)] $\S^1_2(\alpha)\vdash  \varphi(\bar x)\ \to\ \q{$C_\varphi(\bar x,\cdot)$ is an accepting computation of $M_\varphi$ on $\bar x$}$. 
\end{enumerate} 
\end{lemma}

The key lemma is

\begin{lemma}\label{lem:CMT12} If $(N,\mathcal Y)\models \S^1_2(\alpha)+\alpha^c_{M_2}$ for some $c\in\N$, then
$(N,\C_N)\models\T^1_2(\alpha)$.
\end{lemma}

\begin{proof} 
  By  Lemma~\ref{lem:CM} we
  only have to verify~$\Pi^b_1(\alpha)$-induction in~$(N,\C_N)$.
  Let~$\varphi(\bar X,\bar x,y)$ be a~$\Pi^b_1(\alpha)$-formula,
  let~$\bar A$ be parameters from~$\C_N$, let~$\bar a$ be parameters
  from~$N$, and let~$a\in N$. We have to show
$$
(N,\C_N)\models\varphi(\bar A,\bar a,0)\wedge \forall y{<}a\ (\varphi(\bar A,\bar a,y)\to \varphi(\bar A,\bar a,y{+}1))\ \to\ \varphi(\bar A,\bar a,a).
$$

Every set parameter in $\bar A$ equals $\hat C\in\C_N$ for some
$C\in N$. Let $\bar C$ collect these. Using the notation
from the proof of Lemma~\ref{lem:CM}, we have
$\varphi(\bar A ,\bar a,y )^*=\psi(\bar C,\bar a,y)$
where $\psi(\bar z,\bar x,y)\in\Pi^b_1$ and $\varphi(\bar A,\bar a,y)$ is equivalent to 
$\psi(\bar C,\bar a,y)$ in $(N,\C_N)$.
Choose $\chi(u)\in\Pi^b_1$ such that $\S^1_2$ proves
$(\chi(\langle \bar z,\bar x,y\rangle) \leftrightarrow
\psi(\bar z,\bar x,y))$. 
Choose $M_{\chi}$ according to
Lemma~\ref{lem:mc}. By $\Delta^b_1(\alpha)$-comprehension and the choice of $\chi$ we have
$$
(N,\mathcal Y)\models \forall \bar z, \bar x, y\ \big( \exists_2 Y\ 
\q{$Y$ is an accepting computation of $M_{\chi}$ on $\langle \bar z, \bar x, y\rangle$}\leftrightarrow
\psi(\bar z, \bar x, y)\big).
$$
Choose~$n := \max\{|\bar{C}|,|\bar{a}|,|a|,2\}$ in~$N$.
Choose~$e \in \N$ and~$m := n^e$ in~$N$ such that~$N$ satisfies~$\forall \bar{z},\bar{x},y{<}2^n \; |\langle
\bar{z},\bar{x},y\rangle| \leq m$.
Since~$(N,\mathcal Y)\models \alpha^c_{M_2}$, by
Theorem~\ref{thm:equiv2}~(a) there is~$d \in \N$ such
that~$(N,\mathcal Y)\models
\alpha^d_{M_{\chi}}$. Applying~$\alpha^d_{M_{\chi}}$
to~$m$ gives~$D\in N$ such that
$$
(N,\mathcal Y)\models\forall\bar z, \bar x, y{<}2^{n}\ \big(D(\langle
\bar z,\bar x,y\rangle )=1\leftrightarrow \psi(\bar z, \bar x,
y)\big).
$$
But this sentence is number-sort, so we can
equivalently replace~$\mathcal Y$ by~$\C_N$, particularize~$\bar z$
to~$\bar{C}$ and~$\bar x$ to~$\bar{a}$, and then
replace~$\psi(\bar C,\bar a,y)$ back
to~$\varphi(\bar A,\bar a,y)$. Since~$\bar{C},\bar{a},a<2^n$ 
in~$N$, 
$$
(N,\C_N)\models \forall y{\le }a\ \big( E(y)=1\leftrightarrow
\varphi(\bar A,\bar a,y) \big)
$$
where~$E \in N$ is a circuit in $N$ such that  $N\models\forall y{<}2^n\ E(y)=D(\langle \bar C,\bar a,y\rangle)$.  We are left to show
$$
(N,\C_N)\models E(0)=1\wedge \forall y{<}a\ (E(y)=1\to E(y{+}1)=1)\ \to\ E(a)=1.
$$
This is number-sort and, in fact, an instance of $\Delta^b_1$-induction, which holds in $N\models\S^1_2$.
\end{proof}

\begin{proof}[Proof of Lemma~\ref{lem:alphagammaproof}.] 
  We prove only~(a); the proof of~(b) is
  analogous. Assume~$\S^1_2(\alpha)\vdash\alpha^c_{M_1}$ and
  let~$N\models\S^1_2(\alpha)$. We have to
  show~$N\models\meyer^d_{M_1}$ for some~$d\in \N$.

  By Lemma~\ref{lem:CM},~$(N,\C_N)\models \S^1_2(\alpha)$.  By
  assumption,~$(N,\C_N)\models \alpha^c_{M_1}$. By
  Lemma~\ref{lem:alpha} there is~$e\in\N$ such
  that~$(N,\C_N)\models\alpha^e_{M_2}$. By
  Lemma~\ref{lem:CMT12},~$(N,\C_N)\models\T^1_2(\alpha)$.  By
  Lemma~\ref{lem:alphagamma},~$(N,\C_N)\models\meyer^d_{M_1}$ for
  some~$d\in\N$. Being number-sort,~$N\models\meyer^d_{M_1}$ follows.
\end{proof}

\subsection{Consistency from separation}

The following strengthens Theorem~\ref{thm:nbsort}:

\begin{theorem}\label{thm:nbsortM} Assume $N\models\S^1_2+\neg\varphi$ where $\varphi$ is a number-sort sentence.
\begin{enumerate}\itemsep=0pt
\item[(a)] If $\forall\Sigma^{1,b}_1(\V^1_2)\vdash\varphi$, then $N\models\q{$\EXP\not\subseteq\Ppoly$}$.
\item[(b)] If $\forall\Sigma^{1,b}_1(\U^1_2)\vdash\varphi$, then $N\models\q{$\PSPACE\not\subseteq\Ppoly$}$.
\end{enumerate}
\end{theorem}

\begin{proof} We only prove~(a), the proof of~(b) is analogous. It
  suffices to show for all~$N$ and~$c\in\N$:
  if~$N\models\S^1_2+\meyer^c_{M_1}$,
  then~$(N,\C_N)\models\forall\Sigma^{1,b}_1(\V^1_2)$.

  Assume~$N\models\S^1_2+\meyer^c_{M_1}$ and
  let~$\varphi(\bar X,\bar x)$ be a~$\Sigma^{1,b}_1$-formula provable
  in~$\V^1_2$.  Let~$\bar A$ and~$\bar a$ be parameter tuples
  from~$\C_N$ and~$N$.  We have to
  show~$(N,\C_N)\models \varphi(\bar A,\bar a)$.
  Using the notation from the proof of Lemma~\ref{lem:CM}, recall 
    that~$\varphi(\bar A,\bar a)^*$ is obtained by replacing every~$A$
  from~$\bar A$ by~$C(\cdot)$ for some circuit~$C\in N$. Let~$\bar C$
  collect these. Then~$\varphi(\bar A,\bar a)^*=\psi(\bar C,\bar a)$
  for a~$\Sigma^{1,b}_1$-formula~$\psi(\bar y,\bar x)$
  and~$\psi(\bar C,\bar a)$ is equivalent to~$\varphi(\bar A,\bar a)$
  in~$(N,\C_N)$. We claim~$(N,\C_N)\models\psi(\bar C,\bar a)$.

  Since~$\V^1_2$ proves~$\forall_2 \bar X\ \varphi$,
  by~$\Delta^b_1(\alpha)$-comprehension~$\V^1_2$ proves
$$
\chi\ :=\ (\q{$\bar y$ is a tuple of circuits} \to \psi(\bar y,\bar x)).
$$
This is a~$\Sigma^{1,b}_1$-formula without free set
variables. Thus,~$\S^1_2(\alpha)+\meyer_{M_1}^c\vdash\chi$ by
Corollary~\ref{cor:Sigma11cons}~(a).
Thus,~$(N,\C_N)\models\forall \bar y\forall\bar x \ \chi$.
Plugging~$\bar C$ for~$\bar y$ and~$\bar a$ for~$\bar x$ gives our
claim.
\end{proof}

\section{Magnification}\label{sec:magnification}

Write
~$\exists^{\infty}n{\in}\Lg\ \varphi(n,\ldots)$
for~$\forall u\ \exists v\ (|v|>|u|\wedge \varphi(|v|,\ldots))$.

\begin{definition} Let~$c\in\N$ and let~$M$ be an explicit~$\EXP$-
  or~$\PSPACE$-machine~$M$. Define
$$
\io\meyer^{c}_{M}:=\exists^{\infty}n{\in}\Lg\ \exists C{<}2^{n^c}\ \forall x{<}2^n\ \q{$C_x(\cdot)$ is a halting computation of $M$ on $x$}.
$$
Set
$
\q{$\EXP\not\subseteq\ioPpoly$}:=\big\{\neg \io\meyer^{c}_{M_1}\mid c\in\N \big\}
$ and $
\q{$\PSPACE\not\subseteq\ioPpoly$}:=\big\{\neg \io\meyer^{c}_{M_2}\mid c\in\N \big\}.$
\end{definition}
The sentence~$\io\meyer^{c}_{M}$ states that~$M$ is simulated by
size-$n^c$ circuits on unboundedly many input lengths~$n$. The
theories state that for all but boundedly many input lengths~$n$ the
universal machines are not simulated by size-$n^c$ circuits.  

One can
easily improve Theorem~\ref{thm:maincon}~(a):

\begin{corollary}\label{cor:aecons} 
$\S^1_2\cup \q{$\EXP\not\subseteq\ioPpoly$}$ is consistent.
\end{corollary}

\begin{proof} 
  This follows from Theorem~\ref{thm:ioM} below via Takeuti's
  Theorem~\ref{thm:takeuti}. One can also argue directly from
  Theorem~\ref{thm:maincon}~(a) as follows. Inconsistency would mean,
  by compactness, that~$\S^1_2\vdash\io\meyer^{c}_{M_1}$ for
  some~$c\in\N$. It suffices to infer~$\S^1_2\vdash\meyer^d_{M_1}$ for
  some~$d\in\N$.

  The sentence~$\io\meyer^{c}_{M_1}$ has the
  form~$\forall u \exists v \varphi(u,v)$ with~$\varphi(u,v)$
  bounded. By Parikh's
  theorem,~$\S^1_2\vdash \forall u\exists v{<}t(u)\varphi(u,v)$ for
  some term~$t(u)$.  Argue
  in~$\S^1_2$. Let~$n\in\Log$. Choose~$v<t(u)$
  according~$\io\meyer^{c}_{M_1}$ for~$u:=2^n$. Then~$n<m:=|v|\le n^e$
  for suitable~$e\in\N$. Choose~$C<2^{m^c}$ according
  to~$\io\meyer^{c}_{M_1}$ for this~$m$, i.e.,~$C_x(\cdot)$ is a
  halting computation of~$M_1$ on~$x$ for
  all~$x<2^m$. Then~$C<2^{n^{d}}$ for~$d:=ec$ and~$C_x(\cdot)$ is a
  halting computation of~$M_1$ on~$x$ for all~$x<2^n$.
\end{proof}

\begin{remark}\label{rem:void}  
 We discuss almost everywhere lower bounds only for the sake of a magnification result, namely Theorem~\ref{thm:magn}. Observe that the corresponding statement for \q{$\EXP\not\subseteq\Ppoly$} is void. The reason is that this theory is $\Sigma_1$, i.e., its sentences can be written in the form $\exists x\varphi(x)$ with $\varphi(x)$ bounded. If such a theory is true, then  $\S^1_2$ proves it.  
 
 This has been overlooked in \cite{abm}. \end{remark}

We turn to the proof of Theorem~\ref{thm:magn}. The starting point is
a straightforward analogue of Lemma~\ref{lem:alpha}.

\begin{lemma}\label{lem:iogamma} \
\begin{enumerate}\itemsep=0pt
\item[(a)] For every explicit $\EXP$-machine $M$ and every $c\in\N$ there is $d\in\N$ such that 
$$
\S^1_2\vdash (\io\meyer_{M_1}^c\to\io\meyer_{M}^d).
$$ 
\item[(b)] For every explicit $\PSPACE$-machine $M$ and every $c\in\N$ there is $d\in\N$ such that 
$$
\S^1_2\vdash (\io\meyer_{M_2}^c\to\io\meyer_{M}^d).
$$ 
\end{enumerate}
\end{lemma}
\begin{proof} We prove only (a); the proof of (b) is
  analogous. Let~$t_M$ witness~$M$. Argue in~$\T^1_2(\alpha)$ and
  assume~$\io\meyer_{M_1}^c$.  Choose~$e\in\N$ such
  that~$\langle M,x,t_M(x)\rangle<2^{|x|^e}$ for all~$x>0$.
  Let~$n\in\Log$.  By~$\io\meyer_{M_1}^c$ choose~$\ell\in\Log$
  with~$\ell>(n+1)^e$ and such that there is a circuit~$C$ of size at
  most~$\ell^c$ such that~$C_y(\cdot)$ is a halting computation
  of~$M_1$ on~$y$ for all~$y<2^\ell$.

  Let~$m:=\lfloor\ell^{1/e}\rfloor$.  Note that~$m>n$
  because~$\ell > (n+1)^e$. We will use~$m$ to witness~$\io\meyer_M^d$
  beyond~$n$. First note
  that~$\langle M,x,t_M(x)\rangle<2^{m^e}\le 2^\ell$ for
  all~$x<2^m$. Hence, for all~$x<2^m$: 
\begin{center}
$\q{$C_{\langle M,x,t_M(x)\rangle}(\cdot)$ is a halting computation of $M_1$ on $\langle M,x,t_M(x)\rangle$}$.
\end{center}

We are left to find a small circuit~$D$ such that, for all~$x<2^m$,
the set~$D_x(\cdot)$ is a halting computation of~$M$ on~$x$.
Choose~$f\in\N$ such that~$C_{\langle M,x,t_M(x)\rangle}$ has at
most~$m^f$ inputs for all~$x<2^m$,
so~$C_{\langle M,x,t_M(x)\rangle}(\cdot)<2^{m^f}$ for all~$x<2^m$.  By
Lemma~\ref{lem:circuit} choose a circuit~$D$ such that for all~$x<2^m$
and all~$v<2^{m^f}$:
$$
D(x,v)=G_1(C_{\langle M,x,t_M(x)\rangle}(\cdot),x,v),
$$
where~$G_1$ is from Lemma~\ref{lem:M1}~(b).  Then~$D$ works and has size
at most~$m^{d}$, for some~$d\in\N$.
\end{proof}

We observe that Theorem~\ref{thm:nbsortM} generalizes to almost-everywhere lower bounds:

\begin{theorem}\label{thm:ioM} Assume $N\models\S^1_2+\neg\varphi$ where $\varphi$ is a number-sort sentence.
\begin{enumerate}\itemsep=0pt
\item[(a)] If $\forall\Sigma^{1,b}_1(\V^1_2)\vdash\varphi$, then $N\models\q{$\EXP\not\subseteq\ioPpoly$}$.
\item[(b)] If $\forall\Sigma^{1,b}_1(\U^1_2)\vdash\varphi$, then $N\models\q{$\PSPACE\not\subseteq\ioPpoly$}$.
\end{enumerate}
\end{theorem}

\begin{proof} We prove only (a), with (b) being analogous. The proof
  is the same as that of Theorem~\ref{thm:nbsortM}~(a) except that we
  need a version of Corollary~\ref{cor:Sigma11cons}~(a)
  for~$\io\meyer^c_{M_1}$ instead of~$\meyer^c_{M_1}$. Following that
  proof with Lemma~\ref{lem:iogamma} instead of Lemma~\ref{lem:alpha},
  we have to show, for every~$d\in\N$,
  that~$\S^1_2(\alpha)+\io\meyer^d_{M}$ proves
  every~$\varphi(x)\in\hat\Sigma^{1,b}_1$ which is proved by~$\V^1_2$;
  here,~$M$ is chosen for~$\varphi(x)$ according to
  Theorem~\ref{thm:wit}~(a).  This is easy: argue
  in~$\S^1_2(\alpha)+\io\meyer^d_{M}$ and let~$x$ be
  given. Set~$n:=\max\{|x|,2\}$. By~$\io\meyer^d_{M_1}$ there
  are~$n<m\in\Log$ and~$C<2^{m^d}$ such that for all~$y<2^m$, the
  set~$C_y(\cdot)$ is a halting computation of~$M$ on~$y$.
  Then~$Y:=C_x(\cdot)$ is a halting computation of~$M$ on~$x$.  As~$Y$
  exists by~$\Delta^b_1(\alpha)$-comprehension, we get~$\varphi(x)$ by
  Theorem~\ref{thm:wit}~(a).
\end{proof}

\begin{proof}[Proof of Theorem~\ref{thm:magn}] We prove only (a), with
  (b) being
  analogous. If~$\S^1_2\not\vdash\q{$\EXP\not\subseteq\ioPpoly$}$, we
  find~$c\in\N$ and a model~$N\models\S^1_2+\io\meyer^c_{M_1}$.
  Theorem~\ref{thm:ioM} (a) with~$\varphi:=\neg\io\meyer^c_{M_1}$
  implies~$\forall\Sigma^{1,b}_1(\V^1_2)\not\vdash\neg\io\meyer^c_{M_1}$.
\end{proof}

\section{Conclusion}\label{sec:concl}

The central research question is whether~$\S^1_2$ can be proved consistent with~$\NP\not\subseteq\Ppoly$. We
achieved~$\EXP$ in place of~$\NP$ but lowering this to~$\PSPACE$ is
still open:

\begin{theorem*} Is $\S^1_2\cup\q{$\PSPACE\not\subseteq\Ppoly$}$ consistent?
\end{theorem*}
 
This question deserves some interest as a frontier question, also when
understood via Theorem~\ref{thm:nbsort} as a question concerning
independence. Frontier questions can be misleading, of course, but
here our aim is to emphasize one of the weakest open questions whose
answer still requires a new insight. The answer is not sensible to the formalization:
 
 \begin{proposition} The following are equivalent.
 \begin{enumerate}\itemsep=0pt
 \item[(a)] $\S^1_2\cup\{\neg\meyer^c_{M_2}\mid c\in\N\}$ is consistent.
 \item[(b)] $\S^1_2\cup\{\neg\meyer^c_{M}\mid c\in\N\}$ is consistent for some explicit $\PSPACE$-machine $M$.
  \item[(c)] $\S^1_2(\alpha)\cup\{\neg\alpha^c_{M_2}\mid c\in\N\}$ is consistent.
 \item[(d)] $\S^1_2(\alpha)\cup\{\neg\alpha^c_{M}\mid c\in\N\}$ is consistent for some explicit $\PSPACE$-machine $M$.
\end{enumerate}
 \end{proposition}
 
 \begin{proof} The equivalence (a)$\Leftrightarrow$(b) follows from
   Theorem~\ref{thm:equiv2}~(b), the equivalence (c)$\Leftrightarrow$(d)
   from Theorem~\ref{thm:equiv2}~(a), the implication (c)$\Rightarrow$(a)
   from Theorem~\ref{thm:equiv2}~(c). For (a)$\Rightarrow$(c) assume
   (c) fails; by compactness,~$\S^1_2(\alpha)\vdash \alpha^c_{M_2}$
   for some~$c\in\N$; by
   Lemma~\ref{lem:alphagammaproof}~(b),~$\S^1_2\vdash \meyer^d_{M_2}$
   for some~$d$, so (a) fails.
 \end{proof}

 Finally, we note some conditional consistency results based on
 witnessing methods, analo\-gous to the already mentioned ones from
 \cite{cokra}. By~$\PTIME^{\NP}_{\mathrm{tt}}$ we denote
 the class of problems decidable in polynomial time with non-adaptive
 queries to an~$\NP$ oracle.

\begin{proposition}\label{prop:conditional}\
\begin{enumerate}\itemsep=0pt
\item[(a)] $\S^1_2 \cup \q{$\PSPACE\not\subseteq\Ppoly$}$ is consistent if  $\PSPACE\not\subseteq \PTIME^{\NP}_{\mathrm{tt}}$.
\item[(b)] $\S^2_2 \cup \q{$\PSPACE\not\subseteq\Ppoly$}$ is consistent if  
 $\PSPACE\not\subseteq\PTIME^{\NP}$. 
 \item[(c)] $\S^2_2 \cup \q{$\EXP\not\subseteq\Ppoly$}$ is consistent if  
 $\EXP\not\subseteq\PTIME^{\NP}$. 
\end{enumerate}
\end{proposition}

\begin{proof} We start with (b):
  If~$\S^2_2 \cup \q{$\PSPACE\not\subseteq\Ppoly$}$ is inconsistent, then,
  by compactness,~$\S^2_2\vdash\meyer^c_{M_2}$ for
  some~$c\in\N$. Thus,~$\S^2_2$ proves the
  following~$\Sigma_2^b$-formula with variable~$y$: 
  \begin{equation*}\label{eq:etaformula}
|y|>1\to \begin{array}{l}
\exists C{<}2^{|y|^c}\ \forall N,x,t{<}2^{|y|}\ 
 \q{$C_{\langle N,x,t\rangle}(\cdot)$ is a halting computation 
of $M_2$ on $\langle N,x,t\rangle$}.
\end{array}
\end{equation*} 
By Buss' witnessing theorem \cite{bussthesis} for~$\S^2_2$, there is a
function~$f$ computable in polynomial time with an~$\NP$-oracle such
that the set~$f(y)_{\langle N,x,t\rangle}(\cdot)$, i.e., the
set~$\eval(f(y),\langle N,x,t\rangle,\cdot)$, is a halting computation of~$M_2$
on~$\langle N,x,t \rangle$. This implies that the~$\PSPACE$-complete
problem decided by~$M_2$ is in~$\PTIME^{\NP}$.

(c): Proved analogously.

(a): If~$\S^1_2 \cup \q{$\PSPACE\not\subseteq\Ppoly$}$ is inconsistent,
then, by compactness,~$\S^1_2\vdash\meyer^c_{M_2}$ for some~$c\in\N$.
By \cite[Theorem 7.3.5]{krabuch} there is a polynomial time algorithm
that given~$n>1$ in unary and~$N,x,t<2^n$ asks~$O(\log n)$ {\em
  witness queries} (see~\cite[Definition 6.3.6]{krabuch}) to an~$\NP$
oracle and outputs a circuit~$C$ such that~$C_{N,x,t}(\cdot)$ is a
halting computation of~$M_2$ on~$x$. The output~$C$ depends on the
answers to the witness queries but, since~$M_2$ is deterministic, all
outputs~$C$ encode the same computation.  Modify this algorithm to
output 1 or~0 depending on whether this computation is accepting or
rejecting.  By \cite[Corollary 6.3.5]{krabuch} this implies that
the~$\PSPACE$-complete problem decided by~$M_2$ is decidable in
polynomial time with~$O(\log n)$ Boolean queries to an~$\NP$
oracle. Instead, one can also use~$n^{O(1)}$ many non-adaptive queries
(see \cite[Theorem 6.2.3]{krabuch}).
\end{proof}

\begin{remark} 
  The analogue of Proposition~\ref{prop:conditional} for higher
  levels~$\S^i_2$ is void: if~$\PSPACE\not\subseteq\NP^\NP$, then the
  theory \q{$\PSPACE\not\subseteq\Ppoly$} is true
  \cite[Theorem~4.2]{kl}, so consistent with any sound theory; same
  for~$\EXP$ by Meyer's argument.
\end{remark}

To infer the consistency of~$\S^1_2\cup \q{$\EXP
    \not\subseteq \Ppoly$}$ from a witnessing argument along the lines
  of Proposition~\ref{prop:conditional} we would need the
  separation~$\EXP \not\subseteq \PTIME^{\NP}_{\mathrm{tt}}$, which is
  open. It is remarkable that the G\"odel-type argument that underlies
  our proof of consistency of~$\S^1_2 \cup \q{$\EXP \not\subseteq
    \Ppoly$}$ is able bypass this obstacle. We find this encouraging.

\bigskip\noindent\textbf{Acknowledgements} First author was partially
supported by grant PID2022-138506NB- C22 (PROOFS BEYOND) and the
Severo Ochoa and María de Maeztu Program for Centers and Units of
Excellence in R\&D (CEX2020-001084-M) of the AEI, and the CERCA and
ICREA Academia Programmes of the Generalitat de Catalunya.

\end{document}